\documentclass{amsart}

\usepackage{geometry}                
\usepackage{graphicx}
\usepackage{amssymb, amsmath, amscd, amsthm}
\usepackage{epstopdf}
\usepackage{mathdots}
\usepackage[all]{xy}
\usepackage{appendix}
\usepackage[pdftex,bookmarks=true]{hyperref}
\def\C{\mathbb{C}}
\def\Z{\mathbb{Z}}
\def\N{\mathbb{N}}

\def\A{\mathbb{A}}

\def\beqn{\begin{equation}}
\def\eeqn{\end{equation}}
\def\beqnn{\begin{equation*}}
\def\eeqnn{\end{equation*}}
\def\beqna{\begin{eqnarray}}
\def\eeqna{\end{eqnarray}}
\def\beqnan{\begin{eqnarray*}}
\def\eeqnan{\end{eqnarray*}}

\newtheorem{theorem}{Theorem}[section]
\newtheorem{lemma}[theorem]{Lemma}
\newtheorem{proposition}[theorem]{Proposition}

\newtheorem{conjecture}[theorem]{Conjecture}

\newtheorem{definition}[theorem]{Definition}
\newtheorem{remark}[theorem]{Remark}

\theoremstyle{remark}

\numberwithin{equation}{section}

\begin{document}

\title{The Bessel Period of U(3) and U(2) involving a non-tempered representation}

\author{Jaeho Haan}

\address{Algebraic Structure and its Applications Research Center(ASARC), Department
of Mathematics, Korea Advanced Institute of Science and Technology}
\email{jaehohaan@gmail.com}


\begin{abstract}
In \cite{Ha}, Neal Harris has given a refined Gross-Prasad conjecture for unitary group as an analogue of Ichino and Ikeda's paper \cite{Ich} concerning special orthogonal groups. In his paper, he stated a conjecture under the assumption that the pair of given  representations should be tempered. In this paper, we consider a specific pair involving a non-tempered representation of $U(3)$ and suggest its refined period formula restriced to $U(2)$.
\end{abstract}

\maketitle

\section{The Bessel Period of U(3) and U(2) involving a non-tempered representation}\label{intro}

We first recall the Refined Gross-Prasad Conjecture stated in \cite{Ha}. Let $E/F$ be a quadratic extension of number fields and $\A_F$,$\A_E$ are their adele rings respectively. Let $V_{n}\subset V_{n+1}$ be hermitian spaces of dimensions $n$ and $n+1$ over $E$, respectively. Consider the unitary groups $U(V_{n})\subset U(V_{n+1})$ defined over $F$.  Write $G_i := U(V_i)$.  Let $\pi_n$ and $\pi_{n+1}$ be irreducible tempered cuspidal automorphic representations of $G_n(\A_F)$ and $G_{n+1}(\A_F)$ respectively, and we fix isomorphisms $\pi_n \cong \otimes_v \pi_{n,v}$ and $\pi_{n+1} \cong \otimes_v \pi_{n+1,v}$.  We suppose that $\operatorname{Hom}_{G_n(k_v)}(\pi_{n+1,v}\otimes \pi_{n,v},\C)\neq 0$ for every place $v$ of $F$.

  We consider the following $G_{n}(\A_F)\times G_{n}(\A_F)$-invariant functional
$$
\mathcal{P}: (V_{\pi_{n+1}}\boxtimes \bar{V}_{\pi_{n+1}})\otimes ({V}_{\pi_{n}}\boxtimes \bar{V}_{\pi_{n}})\to\C
$$
defined by
\beqn\label{pdef}
\mathcal{P}(\phi_1,\phi_2;f_1,f_2):=\left(\int_{[G_n]} \phi_1(g)f_1(g) dg\right)\cdot\left(\int_{[G_n]} \overline{\phi_2(g)f_2(g)} dg\right)
\eeqn
for $\phi_i\in V_{\pi_{n+1}}$, $f_i\in V_{\pi_n} $ and $[G_n]=G_n(F)\setminus G_n(\A_F).$  If $\phi_1=\phi_2=\phi$ and $f_1=f_2=f$, we simply write $\mathcal{P}(\phi,f):=\mathcal{P}(\phi_1,\phi_2;f_1,f_2)$ and we call $\mathcal{P}$ the global period.

On the other hand, there is another $G_n(\A_F)\times G_n(\A_F)$-invariant functional constructed from the local integral of matrix coefficients. To define matrix coefficients, for each place $v$ of $F$, let $F_v$ be its completion of $F$ at $v$ and denote $G_{i,v}:=G_i(F_v)$.  Fix the local pairings
$$
\mathcal{B}_{\pi_{i,v}}:\pi_{i,v}\otimes\bar{\pi}_{i,v}\to\C
$$
so that
$$
\mathcal{B}_{\pi_i} = \prod_v \mathcal{B}_{\pi_{i,v}}
$$
where $\mathcal{B}_{\pi_i}$ is the Petersson pairing
\beqnan
\mathcal{B}_{\pi_i}(f_1,f_2) &:=& \int_{[G_i]} f_1(g_i)\overline{f_2(g_i)}dg_i
\eeqnan
and the $dg_i$ is Tamagawa measures on $G_i(\A_F)$.
For each place $v$, we define a $G_{n,v}\times G_{n,v}$ invariant functional
$$
\mathcal{P}_v^\natural: (\pi_{n+1,v}\boxtimes \bar{\pi}_{n+1,v})\otimes ({\pi}_{n,v}\boxtimes \bar{\pi}_{n,v})
$$
by
$$
\mathcal{P}_v^\natural(\phi_{1,v},\phi_{2,v}; f_{1,v}, f_{2,v}) := \int_{G_{n,v}} \mathcal{B}_{\pi_{n+1,v}}(\pi_{n+1,v}(g_{n,v})\phi_{1,v},\phi_{2,v})\mathcal{B}_{\pi_{n,v}}(\pi_{n,v}(g_{n,v})f_1,f_2) dg_{n,v}.
$$
Here, the $dg_{n,v}$ are local Haar measures such that $\prod_v dg_{n,v}=dg_n$.

Write $\mathcal{P}_v^\natural(\phi_{v},\phi_{v}; f_{v}, f_{v})=: \mathcal{P}_v^\natural(\phi_v,f_v)$ and we set
\beqnan
\Delta_{G_i} &:=& L(M^\vee_{i}(1),0)\\
\Delta_{G_{i,v}} &:=& L_v(M^\vee_{i}(1),0)
\eeqnan
where $M^\vee_{i}(1)$ is the twisted dual of the motive $M_{i}$ associated to $G_{i}$ by Gross in \cite{Gr}.
It is known in \cite[Prop. 2.1]{Ha} that $\mathcal{P}_v^\natural$ converges absolutely if the $\pi_{i,v}$ is tempered.  Furthermore, it is also known that for unramified data $\phi_v,f_v$ satisfying conditions $(1)-(7)$ in \cite[p.6]{Ha}, we have
$$
\mathcal{P}_v^\natural(\phi_v,f_v) = \Delta_{G_{n+1,v}}\frac{L_{E_v}(1/2,BC(\pi_{n,v})\boxtimes BC(\pi_{n+1,v}))}{L_v(1,\pi_{n,v},\operatorname{Ad})L_v(1,\pi_{n+1,v},\operatorname{Ad})}
$$
 \\(Here, $BC(\pi_i)$ is the quadratic base-change of $\pi_i$ to a representation of $GL_i(\A_E)$)

From this observation, we can normailze $\mathcal{P}_v^\natural$ as  $$
\mathcal{P}_v := \Delta_{G_{n+1,v}}^{-1}\frac{L_v(1,\pi_{n,v},\operatorname{Ad})L_v(1,\pi_{n+1,v},\operatorname{Ad})}{L_{E_v}(1/2,BC(\pi_{n,v})\boxtimes BC(\pi_{n+1,v}))}\mathcal{P}_v^\natural
$$ and call this the local period.

Then $$
\prod_v\mathcal{P}_v: (V_{\pi_{n+1}}\boxtimes \bar{V}_{\pi_{n+1}})\otimes ({V}_{\pi_{n}}\boxtimes \bar{V}_{\pi_{n}})\to\C.
$$ is also another $G_n(\A_F)\times G_n(\A_F)$-invariant functional.

The Refined Gross-Prasad Conjecture predicts that these two global $G_n(\A_F)\times G_n(\A_F)$-functionals  $\mathcal{P}$ and $\prod_v\mathcal{P}_v$ differs by only a certain constant, that is the central $L$-value of the product $L$-function. The precise conjecture is as follows :

\begin{conjecture}[Refined Gross-Prasad Conjecture for Unitary groups]\label{con}
$$
\mathcal{P}(\phi,f) = \frac{\Delta_{G_{n+1}}}{2^\beta}\frac{L_E(1/2,BC(\pi_n)\boxtimes BC(\pi_{n+1}))}{L_F(1,\pi_n,\operatorname{Ad})L_F(1,\pi_{n+1},\operatorname{Ad})}\prod_v \mathcal{P}_v(\phi_v,f_v).
$$
(Here $\psi_i$ is the conjectural $L$-parameter for $\pi_i$ and $\beta$ is an integer such that $2^\beta = |S_{\psi_{n+1}}|\cdot |S_{\psi_n}|$ and $S_{\psi_i}:=\operatorname{Cent}_{\widehat{G_i}}(\operatorname{Im}(\psi_i))$ is the associated component group.)
\end{conjecture}

In \cite{Ha}, N.Harris proved this conjecture unconditionally for $n=1$ using Waldspurger formula, and conditionally for $n=2$ assuming $\pi_3$ is a $\Theta$-lift of a representation on $U(2)$. Recently, Wei Zhang proved for general case using relative trace formula under some local conditions.\cite{Zhang}

Our goal is to provide an analog of this conjecture for $n=2$ and $\pi_3$ is a theta lift of $U(1)$. Note that in this case, $\pi_3$ is no longer tempered and so the above local periods may diverge. So we first regularize the local period using the function appearing in the doubling method. Once this is done, we can define a regularized local period and this enable us to establish the following formula which can be seen as an analogue of Refined Gross-Prasad conjecture. 

\begin{theorem}\label{thm}
Let $F$ be a totally real field and $E$ a totally imaginary quadratic extension of $F$ such that all the finite places of $F$ dividing 2 do not split in $E$. The unitary groups we are considering here are all associated to this extension. Let $\sigma$ be an automorphic characters of $U(1)(\A_F)$ and $\pi_3=\Theta(\bar{\sigma}),\pi_2=\Theta(\bar{\mathbb{I}})$ be irreducible tempered cuspidal automorphic representations of $U(2)(\A_F)$ which comes from a theta lift of $\sigma$ and trivial character $\mathbb{I}$, respectively. We assume that these two theta lifts are nonvanishing and cuspidal. \\Then for $\phi=\otimes \phi_v \in \pi_3$ and $ f=\otimes f_v\in \pi_2$,$$
 \mathcal{P}(\phi,f) = -\frac{1}{2^3}\cdot \frac{L(3,\chi)}{L^2(1,\chi)}\cdot \frac{L_E(\frac{1}{2},BC( \sigma)\otimes \gamma)\cdot Res_{s=0}(L_E(s,BC(\pi_2)\otimes \gamma))}{L_E(\frac{3}{2},BC(\sigma)\otimes \gamma^3)}\prod_v \mathcal{P}_v(\phi_v,f_v).$$
where $\gamma$ is a character of $\A_E^\times/E^\times$ such that $\gamma|_{\A_F^\times}=\chi_{E/F}$ and for $i=1,2$, $\omega_{\pi_i}$ is the central character of $\pi_i$. The normalized local periods $\mathcal{P}_v$'s are defined by $$\mathcal{P}_v(\phi_v,f_v):=c_v\cdot \lim_{s\to 0}\frac{\zeta_v(2s)}{L_v(s,BC(\pi_{2,v})\otimes\gamma_v)}\cdot \int_{U(2)_v}\mathcal{B}_{\pi_{3,v}}(g_{v}\cdot \phi_{v},\phi_{v})\cdot \mathcal{B}_{\pi_{2,v}}(g_v\cdot f_{v},f_v) \cdot \Delta(g_v)^s dg_{v}.$$
(here, $c_v$ is a constant for each $v$ defined by $$c_v:=\frac{L_v^2(1,\chi_{E_v/F_v})\cdot L_{E_v}(\frac{3}{2},BC(\sigma)\otimes\gamma_v^3)}{L_v(3,\chi_{E_v/F_v})\cdot L_{E_v}(\frac{1}{2},BC(\sigma) \otimes \gamma_v)}$$ and  $\mathcal{B}_{\pi_{i,v}}$'s are the fixed local pairings of $\theta(\bar{\sigma})_v$ s.t. $\mathcal{B}_{\pi_i}=\prod_v \mathcal{B}_{\pi_{i,v}}$ and $\Delta(g_v)$ is some function we will define in Section 3.)
\end{theorem}


\begin{remark}In the $SO(n)$ version of the conjecture, Ichino was the first who considered the non-tempered case in \cite{Ich0}, and recently, Yannan Qiu has brought his result into adelic setting including the former.\cite{Qiu}. Thus this article can be considered as an analogue of \cite{Qiu}.
\end{remark}

The rest of the paper is organized as follows: in Section \ref{thetachapt}, we introduce the theta correspondence for unitary groups, as well as the Weil representation. In Section $\ref{rallis}$, we give several versions of the Rallis Inner Product Formula. With all these things put together, we prove Theorem $\ref{thm}$ in Section \ref{finalchapt} under the assumption of the lemma which we shall prove in Section 5.

\section{The $\Theta$-correspondence for Unitary groups}\label{thetachapt}

We review the Weil Representation and $\Theta$-correspondence. Most of this section are excerpts from \cite{Ha}.
\subsection {The Weil Representation for Unitary Groups} \label{Weil}
In this subsection, we introduce the Weil representation. Since the constructiuons of global and local Weil representation are similar, we will treat both of them simultaneously.  For an algebraic group $G$, if the same statement can be applied to both the local and global cases, we will not use the distinguished notation $G(F_v)$ and $G(\A_F)$, but just refer them to $G$. \\ Let $(V,\langle,\rangle_{V})$ and $(W,\langle,\rangle_{W})$ be two hermitian and skew-hermitian spaces of dimension $m,n$ respectively. Denote $G:=U(V)$ and $H:=U(W)$ and we regard them as an algebraic group over $F$.\\

\noindent Define the symplectic space
$$
\mathbb{W} := \operatorname{Res}_{E/F} V \otimes_E W
$$
with the symplectic form $$
\langle v \otimes w,v' \otimes w' \rangle_{\mathbb{W}} := \frac{1}{2}\operatorname{tr}_{E/F}\left(\langle v,v'\rangle_{V}\otimes {\langle w,w' \rangle}_{W}\right).
$$
We also consider the associated symplectic group $Sp(\mathbb{W})$ preserving $\langle \cdot,\cdot \rangle_{\mathbb{W}}$ and the metaplectic group $\widetilde{Sp}(\mathbb{W})$ satisfying the following short exact sequence :

$$1\to \C^\times\to\widetilde{Sp}(\mathbb{W})\to Sp(\mathbb{W})\to 1.$$\\ 
Let $\mathbb{X}$ be a Lagrangian subspace of $\mathbb{W}$ and we fix an additive character $\psi:\A_F/F\to\C^\times$ (globally) or $\psi:F_v\to\C^\times$ (locally). Then we have a Schr\"odinger model of the Weil Representation $\omega_\psi$ of $\widetilde{Sp}(\mathbb{W})$ on $\mathcal{S}(\mathbb{X})$, where $\mathcal{S}$ is the Schwartz-Bruhat function space.\\

Throughout the rest of the paper, let $\chi_{E/F}$ be the quadratic character of $\A_F^{\times}/F^{\times}$ or $F_v^{\times}$ associated to $E/F$ by the global and local class field theory. (For split place $v$, we define $\chi_{E/F}$ the trivial character.) And we also fix some unitary character $\gamma$ of $\A_{E}^{\times}/E^{\times}$ or $E_v^{\times}$ whose restriction to $A_F^{\times}$ or $F_v^{\times}$ is $\chi_{E/F}$.\\

 If we set  \beqnan \gamma_{V} &:=&\gamma^m \\ \gamma_{W}&:=&\gamma^n, \eeqnan
then $(\gamma_{V}, \gamma_{W})$ gives a splitting homomorphism $$
\iota_{\gamma_{V}, \gamma_{W}}:G\times H\to \widetilde{Sp}(\mathbb{W})
$$ and so by composing this to $\omega_\psi$, we have a Weil representation of $G \times H$ on $\mathbb{S}(\mathbb{X})$.

When the choice of $\psi$ and $(\gamma_{V}, \gamma_{W})$ is fixed as above, we simply write $$\omega_{W,V}:= \omega_\psi\circ \iota_{\gamma_{V},\gamma_{W}}.$$

\begin{remark} For $n=1$, the image of $H=U(1)$ in $\widetilde{Sp}(\mathbb{W})$ coincides with the image of the center of $G$, so we can regard the Weil representation of $G\times H$  as the representation of $G$.
\end{remark}

\subsection{The Local $\Theta$-Correspondence}\label{local}  In this subsection, we deal with only the local case and so we suppress $v$ from the notation. (Note that if $v$ is non-split, $E$ is the quadratic extension of $F$ and in the split case, $E=F\oplus F$.) As in previous subsection, for non-split $v$, we denote $\chi_{E/F}$ the quadratic character associated to $E/F$ by local class field theory and for the split case, $\chi_{E/F}$ is trivial. 

\subsubsection{Howe Duality}
Suppose that $(G,G')$ is a dual reductive pair of unitary groups in a symplectic group $Sp(\mathbb{W})$. (Recall that a dual reductive pair $(G,G')$ in $Sp(\mathbb{W})$ is a pair of reductive subgroups of $Sp(\mathbb{W})$ which are mutual centralizers, i.e. $Z_{Sp(\mathbb{W})}(G)=G'$ and $Z_{Sp(\mathbb{W})}(G')=G$.)

After fixing the characters $\psi$ and $\gamma$ as in subsection 2.1, we obtain a Weil representation $(\omega_{\psi,\gamma},\mathcal{S})$ of $G\times G'$. For an irreducible admissible representation $\pi$ of $G$, the maximal $\pi$-isotypic quotient of $\omega$, say $\mathcal{S}(\pi)$, is of the form $$
\mathcal{S}(\pi)\cong\pi\otimes \Theta(\pi).
$$
The \emph{Howe Duality Principle} says that if $\Theta(\pi)$ is nonzero, then
\begin{enumerate}
\item $\Theta(\pi)$ is a finite-length admissible representation of $G'$.
\item $\Theta(\pi)$ has the unique maximal semisimple quotient $\theta(\pi)$ and it is irreducible.
\item The correspondence $\pi\mapsto\theta(\pi)$ gives a bijection between the irreducible admissible representations of $G$ and $G'$ that occur as the maximal semisimple quotients of $\mathcal{S}$.

\end{enumerate}

The third is called the local $\Theta$-correspondence. The Howe duality is now known to hold for all places. (see \cite{Gan0})

\subsection{The Explicit Local Weil representation for $GL(3)(F_v)$} \label{local}
The local Weil representation of unitary groups is explicitly described in \cite{Ha1}. In particular, if $v$ splits, $U(3)(F_v)=\{(A,B)\in M_3(F_v) \ |      \ AB=I\}$ and so by sending $(x,x^{-1})$ to $x$, it is identified to $GL(3)(F_v)$. We record here the explicit local Weil representation of $GL(3)(F_v)$ for later use. \\Let $X=F_v^3$ be a 3-dimensional vector space over $F_v$ with a fixed basis. Then there is a Weil-representation $\omega$ of $GL(3)(F_v)$ realized on $\mathcal{S}(F_v^3)$, which is uniquely determined by the following formula:
\beqna \label{weil}\omega(g)f(x)=\gamma(\det(g))|\operatorname{det}(g)|^{\frac{1}{2}}f(g^tx), \quad \quad x\in F_v^3 \eeqna
Since $E_v=F_v \times F_v$ and $\gamma$, we defined in [\ref{Weil}], is trivial on $F_v$, we can write $\gamma=(\gamma_1,\gamma_1^{-1})$ for some unitary character $\gamma_1$ of $F_v$. Using the above isomorphism of $U(3)$ and $GL(3)$, we can write $\gamma(\operatorname{det}(g))=\gamma_1^2(\operatorname{det}(g)).$ We will use this formula in Section. 5.

\subsection{The Global $\Theta$-Correspondence}

The global $\Theta$-correspondence is realized using $\Theta$-series. To do this, we first define the theta kernel as follows. For any $\varphi\in\mathcal{S}(\mathbb{X}(\A_F))$, let
$$
\theta(g,h,\varphi) := \sum_{\lambda\in \mathbb{X}(F)}\omega_{W,\gamma_W,V,\gamma_V,\psi}(g,h)(\varphi)(\lambda).
$$
Note that this is slowly increasing function. Thus if $f$ is some cusp form on $G(\A_F)$, it is rapidly decreasing and so we can define
\beqn\label{strange}
\theta(f,\varphi)(h) := \int_{[G]} \theta(g,h,\varphi)\overline{f(g)}\ dg
\eeqn
where $dg$ is the Tamagawa measure. \\

\noindent Then the $\Theta$-lift of a cuspidal representation of $G$ as follows:
\begin{definition}
For a cuspidal automorphic representation $\pi$ of $G(\A_F)$, 
$$
\Theta_{V,W,\gamma_W,\gamma_V,\psi}(\pi) = \{\theta(f,\varphi): f\in\pi, \varphi\in\mathcal{S}(\mathbb{X}(\A_F))\}
$$ is called the $\Theta$-lift of $\pi$ with data $(\gamma_W,\gamma_V,\psi).$
\end{definition}

\noindent The Howe Duality Principle implies the following. (\cite{Ge3}, proposition 1.2)

\begin{proposition} If $\Theta(\pi)$ is a cuspidal representation of $U(V)(\A)$, then it is irreducible and is isomorphic to the restricted tensor product $\otimes_v \theta(\pi_v)$.\end{proposition}

\begin{remark}\label{char} Since we integrated $\overline{f}$ (instead of $f$) against the theta series, $\pi$ and $\Theta(\pi)$ have the same central characters.
\end{remark}

\begin{remark}In the theory of theta lift, there are two main issues, that is, the cuspidality and non-vanishing of the theta lift. The cuspidality issue was treated by Rallis in terms of so-called tower property.\cite{R} So to make our Theorem (\ref{thm}) not vacuous, we record the criterion in [\ref{non}] which ensures the non-vanishing of two theta lifts $\pi_3$ and $\pi_2$.
\end{remark}

\section{The Rallis Inner Product Formula}\label{rallis}

The Rallis inner product formula enables us to express the Petersson inner product of the global theta lift with respect to the source information. Since we will need three different version of Rallis inner product formulas, we record them for lifts from $U(1)$ to $U(3)$, $U(1)$ to $U(1)$ and $U(1)$ to $U(2)$. To give a uniform description, we introduce some related notions.

\subsection{Global and Local zeta-integral} \label{zeta}

Let $V$ be a hermitian space over $E$ of dimension $m$, and $W$ be a skew-hermitian space of dimension $n$.  Let $V^-$ be the same space as $V$, but with hermitian form $-\langle\cdot,\cdot\rangle_V$.  Note that $U(V)=U(V^-)$. Let $\tau$ be a irreducible cupspidal automorphic representation of $U(V)$.\\ Denote $G:=U(V)=U(V^-), H:= U(W)$, $G^\diamond := U(V\oplus V^-)$ and $i:G \times G \to G^\diamond$ be the inclusion map  $U(V) \times U(V^-) \hookrightarrow U(V \oplus V^-)$. Let $v$ be a finite place of $F$ and $\mathcal{O}_v$ the ring of integer of $F_v$ and denote by $\varpi$ a generator of its maximal ideal. We fix a maximal compact subgroup $K=\prod_vK_v$ of $G$ such that 
$K_v:=G(\mathcal{O}_v)$ for finite places and $K_v:=G(F_v)\cap U(2m)$ for archimedean places.
Let $P$ be a Siegel-parabolic subgroup of $G^\diamond$ stabilizing $V^{\triangle}:=\{(x,x)\in V \oplus V^-\}$ with Levi-component $GL(V^{\triangle})$ and $\tilde{K}$ a maximal compact subgroup of $G^\diamond$ such that $i(K\times K)\hookrightarrow \tilde{K}$ and $G^\diamond=P\tilde{K}$. Let $I(s,\gamma_W):= \operatorname{Ind}_{P(\A_F)}^{G^\diamond(\A_F)} (\gamma_W\circ\det)\cdot |\det|^s$ be the degenerate principal series representation induced from the character $\gamma_W$ of $\A_E^{\times}$ and $|\det|^s$. (Here, we took $\gamma_W$ as the one we defined in [\ref{Weil}] and the determinants are taken with respect to $GL(V^\Delta)$ which is isomorphic to the Levi of $P$.)

Then for $\Phi_s \in I(\gamma_W,s)$, we define the Eisenstein series

$$
E(\Phi_s,\tilde{g}) := \sum_{x\in P(F)\backslash G^\diamond(F)}\Phi_s(x\tilde{g})
$$
for $\tilde{g}\in G^\diamond$. Then for $f_1, f_2 \in \tau$, we can define

\begin{definition} The Piatetski-Shapiro-Rallis zeta integral is defined as follows:
$$
Z(s,f_1,f_2,\Phi_s,\gamma_W) := \int_{[G\times G]} f_1(g_1)\overline{f_2(g_2)} E(\Phi_s, \iota(g_1, g_2))\gamma_W^{-1}(\operatorname{det}_{U(V^-)} g_2)dg_1 dg_2.
$$
\end{definition}
This integral converges only for $\operatorname{Re}(s)\gg0$. However, once the convergence is ensured, it can be factored into the product of the local-zeta integrals. So we define the local zeta-integrals. Assume that $\Phi_s=\otimes_v \Phi_{s,v}$ and $f_i=\otimes_v f_{i,v}$. Then for each place $v$, the local zeta-integral is defined by 

$$
Z_v(s,f_{1,v}, f_{2,v},\Phi_{s,v}) := \int_{U(V)_v} \Phi_{s,v}(i(g_v, 1)){\langle \pi_v(g_v) f_{1,v},f_{2,v}\rangle_{\pi_v}}dg_v$$ 

We note that the integral defining the $Z_v$ converges for $\operatorname{Re}(s)$ sufficiently large. However, $Z_v$ can be extended to all of $\C$ by meromorphic continuation. For large $s$, there is a factorization theorem of the zeta integral. (See \cite{Ge1} for more detail)
\begin{theorem} \label{3.2}For $\operatorname{Re}(s) \gg 0$,
$$Z(s, f_1, f_2, \Phi_s, \gamma_W) = \prod_v Z_v(s,f_{1,v}, f_{2,v},\Phi_{s,v})$$
\end{theorem}
The local-zeta integral has a simple form for unramified places. Take $S$ to be a sufficiently large finite set of places of $F$ such that for all $v\notin S$, the relevant data is unramified, and the local vectors $f_{i,v}$ are normalized spherical vectors so that $\langle f_{1,v}, f_{2,v}\rangle_{\pi_v}=1$.
Recall that $m= \dim_E V$, $n=\dim_E W$ and set
$$
d_m(s,\gamma_W) := \prod_{r=0}^{m-1} L(2s + m -r,\chi_{E/F}^{n+r}).
$$

It is known that for $v \notin S$, $Z_v$ has the following simple form,

\beqna \label{normal}
Z_v(s,f_{1,v}, f_{2,v},\Phi_{s,v}) =\frac{L_v(s+1/2,\pi\otimes\gamma_W)}{d_{m,v}(s,\gamma_W)}
\eeqna

and so we can normalize them defining $Z_v^{\#}$ by 
\beqnan
Z_v^{\#}(s,f_{1,v}, f_{2,v},\Phi_{s,v}) =\frac{d_{m,v}(s,\gamma_W)}{L_v(s+1/2,\pi\otimes\gamma_W)}\cdot Z_v(s,f_{1,v}, f_{2,v},\Phi_{s,v})
\eeqnan

Thus, we can rewrite Theorem \ref{3.2} as follows:
\\For $f_1,f_2\in\tau$, we have
\beqna \label{normal1}
{Z(s, f_1, f_2, \Phi_s, \gamma_W)} =\frac{L(s+1/2,\pi\otimes\gamma_W)}{d_{m}(s,\gamma_W)}\cdot \prod_v Z_v^{\#}(s,f_{1,v}, f_{2,v},\Phi_{s,v})
\eeqna

\subsection{The Siegel-Weil section}
The Rallis Inner Product Formula relates the Petersson inner product of the global theta lifts to the global zeta-integral for a special  section $\Phi_s \in I(s,\gamma_W)$, so called Siegel-Weil section. In this section, we give the definition of the Siegel-Weil section introducing the doubled Weil representation. 

The setting for the doubled Weil representation is as follows.\\

\noindent We have
$$
\mathbb{W} := \operatorname{Res}_{E/F} 2V\otimes_E W
$$
where $2V := V\oplus V^-$.  We also denote
$$
V^\nabla := \{(v,-v):v\in V\}\subset V\oplus V^-.
$$
Since $V^\nabla \otimes W$ is a Lagrangian subspace of $\mathbb{W}$ over $F$, with some fixed choice of characters $\psi$ and $\gamma$, we have a Schr\"odinger model of the Weil representation $\tilde{\omega}$ of $G^\diamond \times H$ realized on $\mathcal{S}((V^\nabla\otimes W)).$

Now, fix polarizations
\beqnan
V &=& X^+\oplus Y^+\\
V^- &=& X^-\oplus Y^-
\eeqnan
and denote
\beqnan
X &:=& X^+\oplus X^-\\
Y &:=& Y^+\oplus Y^-.
\eeqnan
Then 
$$
2V = X\oplus Y
$$
and so we have another Lagrangian $X\otimes W$ of $\mathbb{W}$.

\noindent If we set \beqnan
\mathbb{X} &:=& X\otimes W\\
\mathbb{X}^+ &:=& X^+\otimes W\\
\mathbb{X}^- &:=& X^-\otimes W,
\eeqnan
then there is a $U(V)(\A_F)\times U(V^-)(\A_F)$-intertwining map
$$
\rho_{m,n}: \mathcal{S}(\mathbb{X}^+(\A_F))\otimes\mathcal{S}(\mathbb{X}^-(\A_F))\to \mathcal{S}(\mathbb{X}(\A_F))\to \mathcal{S}((V^\nabla\otimes W)(\A_F))
$$
where the first map is the obvious one, and the second map is given by the Fourier transform. Furthermore, it satisfies $\rho_{m,n}(\varphi_1\otimes \bar{\varphi_2})(0)=<\varphi_1, \varphi_2>$ and so $\big(\tilde{\omega}(i(g,1))\cdot \rho_{m,n}(\varphi_1\otimes \bar{\varphi_2)}\big)(0)=<\omega_{W,V}(g)\cdot\varphi_1, \varphi_2>$.(\cite[p.182]{Li}) Let $s_m=\frac{n-m}{2}$. By the explicit formula for $\tilde{\omega}$ described in \cite{Ku2}, there is an intertwining map $[\ ]: \mathcal{S}(V^\nabla\otimes W) \to I(s_m,\gamma_W)$ given by $\Phi \to f_\Phi^{s_m}(\tilde{g})=\tilde{\omega}(\tilde{g})\Phi(0)$. We can also extend $f_\Phi^{s_m}$ to $f_\Phi^{s} \in I(s,\gamma_W)$ for all $s \in \C$ by defining $f_\Phi^s:= f_\Phi^{s_m}\cdot |\det|^{s-s_m}$ and call this the Siegel-Weil section in $I(s,\gamma_W)$. (Here the determinant map was taken as in \ref{zeta}.) Then we can define the function $\Delta_m$ of $G$ as $\Delta_m(g):=|\det(i(g,1))|$ and using $\Delta_m$, we can write the Siegel-Weil section as
\label{1}\beqna
f_{[\rho_{m,n}(\varphi_1\otimes \bar{\varphi_2})]}^s(g)=\langle \omega_{W,V}(g)\cdot\varphi_1, \varphi_2\rangle\cdot \Delta_m(g)^{s-s_m}.
\eeqna

Note that $\Delta_m(g)$ is $K\times K$ invariant and $\Delta_m(1)=1$. (For $k_1, k_2 \in K$, $(k_1gk_2,1)=(k_1,k_1)\cdot(g,1)\cdot(k_2,k_1^{-1})$ and $(k_1,k_1)\in P$, $(k_2,k_1^{-1})\in \tilde{K}$.)
Using the similar argument of Prop.6.4 in \cite{Ge1}, Yamana[\cite{Ya1}, Lemma A.4.] computed $\Delta_m(g_v)$ explicitly for split place $v$ of $F$ not dividing 2. We record his computation for splitting places.\\

(1) $v:$ non-archimedean place : Let $v$ be a finite place of $F$ which splits in $E$ and not divide 2. Let $\mathcal{O}_v$ be the ring of integer of $F_v$ and $\varpi$ a generator of its maximal ideal. Since $v$ splits, $U(m)(F_v)\simeq GL(m)(F_v)$ and by Cartan decomposition, $ GL(m)(F_v)=K_mD_m^+K_m$ where $K_m=GL(m)(\mathcal{O}_v)$ and $D_m^+=\operatorname{diag}[\varpi^{a_1}, \cdots , \varpi^{a_m}]$. Then,
\beqna \label{height} \Delta_m(g_v)=|\varpi|^{\sum_{i=1}^m |a_i|}
\eeqna

(2) $v:$ archimedean place : Let $v$ be a finite place of $F$ which splits in $E$ and let $[F_v:\mathbb{R}]=2$. Write $K_m=\{g\in GL(m)(F_v) \ |\  ^{t}\bar{g}g=I_n\}$. If $g=k_1dk_2$ with $k_1,k_2\in K_m$ and $d=\text{diag}[d_1,d_2,\cdots,d_m]$ with positive and reals $d_i$, then \beqna \label{height2} \Delta_m(g_v)=2^{mts}\prod_{i=1}^{m}(d_i^{-1}+d_i)^{-ts}.\eeqna
\begin{remark} \label{3.3}Since $|a+b|\ne|a|+|b|$, we cannot expect $\Delta_m(g_vl_v)\ne\Delta_m(g_v)\Delta_m(l_v)$ for central diagonal matrix $l_v=\operatorname{diag}[\varpi^{c}, \cdots , \varpi^{c}]\in GL(m)(F_v)$.

\end{remark}

Now, we are ready to state the three versions of Rallis Inner Product formula. The first one is as follows;

\subsection{Lifting from $U(1)$ to $U(3)$} Here, $\operatorname{dim}V=1$, $\operatorname{dim}W=3$ and $\tau$ is a irreducible automorphic representation of $U(1)(\A_F)$. Suppose that $f_i=\otimes_v f_{i,v} \in \tau, \mbox{  }\varphi_1=\otimes_v \varphi_{1,v}\in \mathcal{S}(\mathbb{X}^+(\A_F))$ and $\varphi_2=\otimes_v \varphi_{2,v}\in \mathcal{S}(\mathbb{X}^-(\A_F))$. Let $\Phi_{s,v}\in  I(s,\gamma^3)$ is a holomorphic Siegel-Weil section given by $[\rho_{1,3}(\varphi_1\otimes\bar{\varphi_2})]$. Then,
\begin{theorem}\label{thm3.4}
$$
{\langle \theta(\bar{f}_1, \varphi_1), \theta(\bar{f}_2, \varphi_2)\rangle_{\Theta(\bar\tau)}} = \frac{L_E(\frac{3}{2}, BC(\tau)\otimes \gamma^3)}{L(3,\chi_{E/F})}\prod_v Z_v^\sharp(1, {f}_{1, v}, {f}_{2, v}, \Phi_{1, v})
$$
where
$$
Z_v^\sharp := \frac{L_v(3,\chi_{E_v/F_v})}{L_{E_v}(\frac{3}{2}, BC(\tau_v)\otimes \gamma_v^3)}\cdot Z_v
$$
\end{theorem}
\begin{proof}
This follows immediately from Theorem 2.1 in \cite{Li} and (\ref{normal1}) the normalization of the local-zeta integral.
\end{proof}

The next following two versions of Rallis Inner product formula come from Lemma 10.1 in \cite{Ya}:

\subsection{Lifting from $U(1)$ to $U(1)$}\label{1,1}Here, $\operatorname{dim}V=\operatorname{dim}W=1$ and $\tau$ is a irreducible automorphic representation of $U(1)(\A_F)$. Suppose that $f_i=\otimes_v f_{i,v} \in \tau, \mbox{  }\varphi_1=\otimes_v \varphi_{1,v}\in \mathcal{S}(\mathbb{X}^+(\A_F))$ and $\varphi_2=\otimes_v \varphi_{2,v}\in \mathcal{S}(\mathbb{X}^-(\A_F))$. Let $\Phi_{s,v}\in  I(s,\gamma)$ is a holomorphic Siegel-Weil section given by $[\rho_{1,1}(\varphi_1\otimes\bar{\varphi_2})]$. By \cite[Theorem 4.1] {Ya} and \ref{normal1} , we have
\begin{theorem}\label{3.5}
$$
{\langle \theta(\bar{f}_1, \varphi_1), \theta(\bar{f}_2, \varphi_2)\rangle_{\Theta(\bar\tau)}} = \frac{1}{2}\cdot\frac{L_E(\frac{1}{2}, BC(\tau)\otimes \gamma)}{L(1,\chi_{E/F})}\prod_v Z_v^\sharp(0, {f}_{1, v}, {f}_{2, v}, \Phi_{0, v})
$$
where
$$
Z_v^\sharp = \frac{L_v(1,\chi_{E_v/F_v})}{L_{E_v}(\frac{1}{2}, BC(\tau_v)\otimes \gamma_v)} \cdot Z_v
$$
\end{theorem}

\subsection{Lifting from $U(2)$ to $U(1)$} Here, $\operatorname{dim}V=2$, $\operatorname{dim}W=1$ and $\tau$ is a irreducible automorphic representation of $U(2)(\A_F)$. Suppose that $f_i=\otimes_v f_{i,v} \in \tau, \mbox{  }\varphi_1=\otimes_v \varphi_{1,v}\in \mathcal{S}(\mathbb{X}^+(\A_F))$ and $\varphi_2=\otimes_v \varphi_{2,v}\in \mathcal{S}(\mathbb{X}^-(\A_F))$. Let $\Phi_{s,v}\in  I(s,\gamma)$ be a holomorphic Siegel-Weil section given by $[\rho_{2,1}(\varphi_1\otimes\bar{\varphi_2})]$. Then,
\begin{theorem}\label{3.6}
$$
{\langle \theta(\bar{f}_1, \varphi_1), \theta(\bar{f}_2, \varphi_2)\rangle_{\Theta(\bar\tau)}} =  \frac{-Res_{s=0}(L_E(s, BC(\tau)\otimes \gamma))}{L(1,\chi_{E/F})}\prod_v Z_{v,s=-\frac{1}{2}}^\sharp(s, {f}_{1, v}, {f}_{2, v}, \Phi_{s, v})
$$
where
\beqnan
Z_{v,s=-\frac{1}{2}}^\sharp(s, {f}_{1, v}, {f}_{2, v}, \Phi_{s, v})=
\lim_{s\to 0} \frac{L_v(2s+1,\chi_{E_v/F_v})\cdot \zeta_v(2s)}{L_{E_v}(s, BC(\tau_v)\otimes \gamma_v)} \cdot Z_{v}(s-\frac{1}{2}, {f}_{1, v}, {f}_{2, v}, \Phi_{s-\frac{1}{2}, v})
\eeqnan
\end{theorem}

\begin{proof}
By Lemma 10.1 (2) in \cite{Ya} and (\ref{normal1}), $$
{\langle \theta(\bar{f}_1, \varphi_1), 
\theta(\bar{f}_2, \varphi_2)\rangle_{\Theta(\bar\tau)}} = \frac{1}{2}\cdot \lim_{s\to0} \frac{L_E(s, BC(\tau)\otimes \gamma)}{L(2s+1,\chi_{E/F})\zeta_F(2s)}\prod_v Z_{v}^\sharp(s-\frac{1}{2}, {f}_{1, v}, {f}_{2, v}, \Phi_{s-\frac{1}{2}, v}). $$   By Theorem 9.1 and Lemma 10.2 in \cite{Ya}, if $\theta(\bar{\tau})$ doesn't vanish, $L_E(s, BC(\tau)\otimes \gamma)$ has a simple pole at $s=0$. Note that $\zeta_F(s)$ is the completed Dedekind zeta function of $F$ and it has a simple pole at $s=0$. Since $Res_{s=0}\zeta_F(s)=-1$ and $L(1,\chi_{E/F})$ is nonzero, we get $$\lim_{s\to0} \frac{L_E(s, BC(\tau)\otimes \gamma)}{L(2s+1,\chi_{E/F})\zeta_F(2s)}=\frac{- Res_{s=0}(L_E(s, BC(\tau)\otimes \gamma))}{L(1,\chi_{E/F})}. $$ For each $v$, $d_2(s-\frac{1}{2},\gamma_W)\cdot \Phi_{s-\frac{1}{2}, v}(g)$ is not holomorphic but good section (see, \cite{Ya}), so  by Theorem 5.2 in\cite{Ya}, the quotient of $L_v(2s+1,\chi_{E_v/F_v}) \cdot \zeta_v(2s)\cdot Z_{v}(s-\frac{1}{2}, {f}_{1, v}, {f}_{2, v}, \Phi_{s-\frac{1}{2}, v})$ by $L_{E_v}(s, BC(\tau_v)\otimes \gamma_v)$ is holomorphic. Thus each $Z_{v,s=-\frac{1}{2}}^\sharp(s, {f}_{1, v}, {f}_{2, v}, \Phi_{s, v})$ exists and it proves theorem when $\theta(\bar{\tau})$ is nonvanishing. When $\theta(\bar{\tau})$ is zero, then $L_E(s, BC(\tau)\otimes \gamma)$ is holomorphic by Lemma 10.2 in \cite{Ya}, and so $Res_{s=0}(L_E(s, BC(\tau)\otimes \gamma)$ is zero. So the theorem also holds in this case.
\end{proof}

\subsection{The local-to-global criterion for the non-vanishing of the theta lifts}\label{non} Since we will assume $\pi_3$ and $\pi_2$ are non-vanishing, we descrive the non-vanishing criterion of the theta lifts $\pi_3, \pi_2$ as well as from $U(1)$ to $U(1)$.

\subsubsection{Theta lift from $U(1)$ to $U(3)$}

Let $\tau$ be a character of $U(1)$. By the [Lemma 5.3 , \cite{Li}], the Euler product $L_E(s, BC(\tau)\otimes \gamma^3)$ absolutely converges and nonzero at $s=\frac{3}{2}$. Then by (\ref{thm3.4}), we see that $\pi_3=\Theta(\bar{\tau})$ does not vanish when the local zeta integral $Z_v(1, \cdot) \in Hom(I(1,\gamma_v^3) \otimes \tau_v^{\lor} \otimes \tau_v)$ is nonzero for all the places $v$.

\subsubsection{Theta lift from $U(1)$ to $U(2)$}{\label{12}} Let $\tau$ be a character of $U(1)$. Then by [Theorem 5.10, \cite{Ha}], the theta lift $\pi_3=\Theta(\bar{\tau})$ does not vanish when $L_E(1,BC(\tau)\otimes \gamma^2)\ne0$ and local theta lift $\theta_v(\bar{\tau_v})\ne0$ for all the places $v$.

\subsubsection{Theta lift form $U(1)$ to $U(1)$} Let $V$ (resp, $W$) be a hermitian (resp, skew-hermitian) space of dimension 1 over $E$. Let $\tau$ be a character of $U(V)(\A_F)$. Then by (\ref{1,1}) and [Theorem 6.1, \cite{Ha1}], the theta lift $\Theta(\bar{\tau})$ is non-vanishing if and only if $L_E(\frac{1}{2}, BC(\tau)\otimes \gamma)\ne 0$ and for all $v$, $\epsilon_v(\frac{1}{2},\tau_v\otimes \gamma_v, \psi_v)= \epsilon_V\cdot \epsilon_W$. (Here, $\epsilon_V(s,\cdot),\epsilon_W(s,\cdot)$ are the local root number and $\epsilon_v$ is the sign of $V_{E_v}$,$W_{E_v}$ respectively.)

\section{Proof of Theorem 1.2}\label{finalchapt}We remind the reader of our setting.

\subsection{The Setup}$F$ is a totally real number field and $E$ a totally imaginary quadratic extension of $F$.\begin{footnote}{Indeed, this assumption is not essential. See Remark \ref{5.2}} \end{footnote}
We consider the following seesaw diagram:
\beqn\label{mainss}
\xymatrix{U(V\oplus L)\ar @{-}[d]\ar @{-}[dr]& U(W)\times U(W)\ar @{-}[dl]\ar @{-} [d]\\ U(V)\times U(L) & U(W)}
\eeqn

(Here, $V$ is a $2$-dimensional hermitian space over $E/F$ and $W$ is a $1$-dimensional skew-hermitian space over $E/F$ and $L$ is a hermitian line over $E/F$. \\ 
Using the seesaw duality, we can relate the period integral in Theorem to the triple product integral over $U(W)$. 

We first fix the following:
\begin{itemize}
\item $\pi_2=\otimes \pi_{2,v}$ is an irreducible, cuspidal, tempered, automorphic representation of $U(V)(\A_F).$
\item $\sigma=\otimes \sigma_v$ is an automorphic character of $U(W)(\A_F).$
\item $\mu := w_{\pi_{2}}^{-1}\cdot \sigma$ is an automorphic character of $U(L)(\A_F)$, where $\omega_{\pi_2}$ is the central character of $\pi_2$ and $\mu=\otimes \mu_v$ where $\mu_v=w_{\pi_{2,v}}^{-1}\cdot \sigma_v$.
\item $(\omega_{V\oplus L,W},\psi)$ is a Weil representation of $\widetilde{Sp}(\mathbb{W})(\A_F)$.  (See Chapter \ref{thetachapt} for notation.)
\end{itemize}
We also fix local pairings $\mathcal{B}_{\pi_{2,v}}, \mathcal{B}_{\sigma_v}, \mathcal{B}_{\mu_v}$ such that  $\prod_v \mathcal{B}_{\pi_{2,v}}, \prod_v\mathcal{B}_{\sigma_v}$, $\prod_v\mathcal{B}_{\mu_v}$ give the respective Petersson inner products on the global representation and $\mathcal{B}_{\mu_v}(\mu_v,\mu_v)=\mathcal{B}_{\sigma_v}(\sigma_v,\sigma_v)$ for all places $v$. (Since $\mathcal{B}_{\sigma}(\sigma,\sigma)=\mathcal{B}_{\mu}(\mu,\mu)=\operatorname{Vol}([U(1)])$, these choices can stand with no conflict.)\\

\noindent We take $\gamma_L,\gamma_W = \gamma$ and $\gamma_V=\gamma^2$, where $\gamma$ is a unitary character of $\A_E^\times/E^\times$ such that $\gamma|_{\A_F^\times} = \chi_{E/F}$and fix additive character $\psi:\A_F \to \C.$ After fixing these splitting data $(\gamma_V,\gamma_L,\gamma_W,\psi)$, we can define the relevent theta lifts and denote them $\Theta(\bar\pi_2):=\Theta_{W,V\gamma_W,\gamma_V,\psi}(\bar\pi_2)$ on $U(W)(\A_F)$, $\Theta(\bar\sigma):=\Theta_{W,V\oplus L,\gamma_W\gamma_V,\gamma_L,\psi}(\bar\sigma)$ on $U(V\oplus L)(\A_F)$, and $\Theta(\bar\mu):=\Theta_{W,L\gamma_W,\gamma_L,\psi}(\bar\mu)$ on $U(W)(\A_F)$. We assume that all $\Theta$-lifts we consider here are non-vanishing and cuspidal.
\subsection{Proof of Theorem \ref{thm}} In the course of the proof, we will regard $\mu$ and $\sigma$ as automorphic forms in the 1-dimension representations of $\mu$ and $\sigma$ and take $f_{\mu}=\mu$, and $f_{\sigma}=\sigma$. Since $\omega_{W,V\oplus L}=\omega_{W,V} \otimes \omega_{W,L}$, we prove the theorem assuming $\varphi=\varphi_1 \otimes \varphi_2$ for $\varphi_1\in \omega_{W,V}$ and $\varphi_2\in \omega_{W,L}$.\\

$Step\ 1.$ First, we consider another the global period
$$
\mathcal{P}': V_{\Theta(\bar\sigma)}\otimes {V}_{\pi_2} \otimes {V}_\mu \to \C
$$
defined by \beqnan
\mathcal{P}'(f_{\Theta(\bar\sigma)}, f_{\pi_{2}}, f_{\mu}) :=\left| \int_{[U(V)\times U(L)]} f_{\Theta(\bar\sigma)}(i(g,l)){f_{\pi_2}(g) f_\mu(l)}dgdl\right|^2.
\eeqnan

(Here, $i$ is the natural embedding $i:U(V)\times U(L)\hookrightarrow U(V\oplus L)$.) \\By making a change of variables $g \to gl$, we see that $$\int_{[U(V)\times U(L)]} f_{\Theta(\bar\sigma)}(i(g,l)){f_{\pi_2}(g) f_\mu(l)}dgdl=\int_{[U(V)\times U(L)]} f_{\Theta(\bar\sigma)}(i(gl,l)){f_{\pi_2}(gl) f_\mu(l)}dgdl.$$

By Remark \ref{char}, the central character of $\Theta(\bar{\sigma})$ is $\omega_\sigma^{-1}=\sigma^{-1}$.  So, after observing that $(l,l)$ is in the center of $U(V\oplus L)$ and $l$ is in the center of $U(V)$, we have
\beqnan
&&\int_{[U(V)\times U(L)]} f_{\Theta(\bar\sigma)}(i(gl,l)){f_{\pi_2}(gl)\mu(l)dgdl}\\ 
&&= \int_{[U(V)\times U(L)]} \omega_{\Theta(\bar\sigma)}(l)\omega_{\pi_2}(l)\mu(l) f_{\Theta(\bar\sigma)}|_{U(V)}(g) {f_{\pi_2}(g)}dgdl\\ 
&&= \int_{[U(V)\times U(L)]} f_{\Theta(\bar\sigma)}|_{U(V)}(g) {f_{\pi_2}(g)}dgdl\\
&&= \operatorname{Vol}([U(L)]) \int_{[U(V)]} f_{\Theta(\bar\sigma)}|_{U(V)}(g) {f_{\pi_2}(g)}dg\\
&&= 2\int_{[U(V)]} f_{\Theta(\bar\sigma)}|_{U(V)}(g) {f_{\pi_2}(g)}dg. \quad \text{ (note that} \operatorname{Vol}([U(1)])=2)
\eeqnan

Thus, we get $\mathcal{P}(f_{\Theta(\bar{\sigma})},f_{\pi_2})=\frac{1}{4}\mathcal{P}'(f_{\Theta(\bar\sigma)}, f_{\pi_2}, f_{\mu})$. \\

$Step\ 2.$ By the global seesaw duality, we see that $$\int_{[U(V)\times U(L)]} \theta(\bar{{\sigma}},\varphi)(i(g,l)){f_{\pi_2}(g) \mu(l)}dgdl=\int_{[U(W)]} \theta(\bar{f_{\pi_2}},\varphi_1)(h) \theta(\bar{{\mu}},\varphi_2)(h){\sigma}(h)dh$$ (The order change of integration is justified by the rapidly decreasing property of cusp forms and the moderate growth of the theta series.)\\
Since $\Theta(\bar{\pi_2})$ and $\Theta(\bar{\mu})$ have central characters $\omega_{\pi_2}^{-1}$ and $\mu^{-1}$ respectively, we see that $$\mathcal{P}'(\theta(\bar{f_{\pi_2}},\varphi_1)), \theta(\bar{f_{\mu}},\varphi_2)), f_{\sigma})=|\theta(\bar{f_{\pi_2}},\varphi_1)(1)\theta(\bar{\mu},\varphi_2)(1)\sigma(1)|^2\cdot\operatorname{Vol}([U(W)])^2.$$ For $\tau=\pi_2$ or $\mu$ and $i=1,2$  , $$\mathcal{B}_{\Theta(\bar{\tau})}(\theta(\bar{f_{\tau}},\varphi_i),\theta(\bar{f_{\tau}},\varphi_i))=|\theta(\bar{f_{\tau}},\varphi_i)(1)|^2\cdot \operatorname{Vol}([U(W)]) \mbox{  and  } \sigma(1)=1.$$ Thus we can write
\beqnan
\mathcal{P}'(\theta(\bar{f_{\pi_2}},\varphi_1)), \theta(\bar{f_{\mu}},\varphi_2)), f_{\sigma})=\mathcal{B}_{\Theta(\bar{\pi_2})}(\theta(\bar{f_{\pi_2}},\varphi_1),\theta(\bar{f_{\pi_2}},\varphi_1))\cdot \mathcal{B}_{{\Theta(\bar{\mu})}}(\theta(\bar{f_{\mu}},\varphi_2),\theta(\bar{f_{\mu}},\varphi_2)).
\eeqnan By theorem \ref{3.5} and \ref{3.6}, we see that $\mathcal{P}'(\theta(\bar{f_{\pi_2}},\varphi_1)), \theta(\bar{f_{\mu}},\varphi_2)), f_{\sigma})=$ $$-\frac{1}{2}\cdot \frac{L_E(\frac{1}{2}, BC(\mu)\otimes \gamma)}{L(1,\chi_{E/F})}\cdot \frac{Res_{s=0}(L_E(s, BC(\pi_2)\otimes \gamma))}{L(1,\chi_{E/F})}\cdot \prod_v Z_v^{\sharp}(f_{\mu_v},f_{\pi_{2,v}},\varphi_{1,v},\varphi_{2,v})$$
where $Z_v^{\sharp}(f_{\pi_{2,v}},f_{\mu_v},\varphi_{1,v},\varphi_{2,v})=Z_{v,s=-\frac{1}{2}}^\sharp(s, {f}_{\pi_{2,v}}, {f}_{\pi_{2,v}}, \Phi_{s, v})\cdot Z_v^\sharp(0, {f}_{\mu_v}, {f}_{\mu_v}, \Phi_{0, v})$
and \\$\Phi_{s, v}=[\rho_{2,1}(\varphi_1\otimes\bar{\varphi_1})]\in I(s,\gamma)$, $\Phi_{0, v}=[\rho_{1,1}(\varphi_2\otimes\bar{\varphi_2})]\in I(0,\gamma)$.\\
(Note that $ Z_v^{\sharp}(f_{\pi_{2,v}},f_{\mu_v},\varphi_{1,v},\varphi_{2,v})=1$ for unramified data)\\

$Step\ 3.$ Let us make use several abbreviations for various matrix coefficients.
\beqnan&& \mathcal{B}_{\omega_{W, V}}^{\varphi_{1,v}}(g_v):=\mathcal{B}_{\omega_{W, V}}(\omega_{W,V}(g_v,1)\cdot\varphi_{1,v}, \varphi_{1,v})\mbox{ , } \mathcal{B}_{\omega_{W, L}}^{\varphi_{2,v}}(l_v):=\mathcal{B}_{\omega_{W,L}}(\omega_{W, L}(1,l_v)\cdot\varphi_{2,v},\varphi_{2,v}),\\&& \mathcal{B}_{\omega_{W,V\oplus L}}^{\varphi_v}(g_v,l_v):=\mathcal{B}_{\omega_{W,V\oplus L}}(\omega_{W,V\oplus L}(i(g_v,1),l_v)\varphi_v,\varphi_v)=\mathcal{B}_{\omega_{W,V\oplus L}}(\omega_{W,V\oplus L}(i(g_vl_v,l_v),1)\varphi_v,\varphi_v), \\&& \mathcal{B}_{\pi_{2,v}}^{ f_{\pi_{2,v}}}(g_v):=\mathcal{B}_{\pi_{2,v}}(g_v\cdot   f_{\pi_{2,v}},f_{\pi_{2,v}})\mbox{  ,  }  \mathcal{B}_{\tau_v}^{ f_{\tau_v}}(l_v):=\mathcal{B}_{\tau_v}(l_v\cdot f_{\tau_v},f_{\tau_v}) \mbox{ for $\tau=\sigma$ or $\mu$.} \eeqnan
(Recall that in the Weil representation $\omega_{W,V\oplus L}$, the elements in $U(W)$ act as the central element in $U(V\oplus L)$.)

If we unfold  $Z_{v,s=-\frac{1}{2}}^\sharp(s, {f}_{\pi_{2,v}}, {f}_{\pi_{2,v}}, \Phi_{s, v})$ in $Z_v^{\sharp}(f_{\pi_{2,v}},f_{\mu_v},\varphi_{1,v},\varphi_{2,v})$, we can write \\$Z_v^{\sharp}(f_{\pi_{2,v}},f_{\mu_v},\varphi_{1,v},\varphi_{2,v})=$
$$\lim_{\Re(s)\to 0+} \frac{L_v(2s+1,\chi_{E_v/F_v})\cdot \zeta_v(2s)}{L_{E_v}(s, BC(\pi_{2,v})\otimes \gamma_v)} \cdot \int_{U(V)_v} Z_v^\sharp(0, {f}_{\mu_v}, {f}_{\mu_v}, \Phi_{0, v})\mathcal{B}_{\omega_{W,V}}^{\varphi_{1,v}}(g_v) \mathcal{B}_{\pi_{2,v}}^{ f_{\pi_{2,v}}}(g_v)\Delta_2(g_v)^{s}dg_v$$
$$=\frac{L_v^2(1,\chi_{E_v/F_v})}{L_{E_v}(\frac{1}{2}, BC(\mu_v)\otimes \gamma_v)} \cdot  \lim_{\Re(s)\to 0+} \frac{\zeta_v(2s)}{L_{E_v}(s, BC(\pi_{2,v})\otimes \gamma_v)}\cdot I_v(s,\varphi_{1,v}, \varphi_{2,v},f_{\pi_{2,v}},f_{\mu_v}) \text{   where  }$$ \beqnan&&I_v(s,\varphi_{1,v}, \varphi_{2,v},f_{\pi_{2,v}},f_{\mu_v}):=\\&&\int_{U(V)_v}\left(\int_{U(L)_v} \mathcal{B}_{\omega_{W,L}}^{\varphi_{2,v}}(l_v)\cdot \mathcal{B}_{\mu_v}^{ f_{\mu_v}}(l_v)dl_v \right) \cdot \mathcal{B}_{\omega_{W,V\oplus L}}^{\varphi_{1,v}}(g_v)\cdot \mathcal{B}_{\pi_{2,v}}^{ f_{\pi_{2,v}}}(g_v)\cdot \Delta_2(g_v)^{s}dg_v.\eeqnan

Set $J(s,g_v,l_v,\varphi_{1,v}, \varphi_{2,v},f_{\pi_v},f_{\mu_v}):=\mathcal{B}_{\omega_{W,L}}^{\varphi_{2,v}}(l_v)\cdot \mathcal{B}_{\mu_v}^{ f_{\mu_v}}(l_v) \cdot \mathcal{B}_{\omega_{W,V\oplus L}}^{\varphi_{1,v}}(g_v)\cdot \mathcal{B}_{\pi_{2,v}}^{ f_{\pi_{2,v}}}(g_v)\cdot \Delta_2(g_v)^{s}.$
Then we can write $I_v$ as a double integral, $$I_v(s,\varphi_{1,v}, \varphi_{2,v},f_{\pi_{2,v}},f_{\mu_v})=\int_{U(V)_v \times U(L)_v} J(s,g_v,l_v,\varphi_{1,v}, \varphi_{2,v},f_{\pi_{2,v}},f_{\mu_v}) dg_v dl_v.$$ Since $\pi_2$ is tempered, by Lemma 7.2 in \cite{Ya}, $Z_v(s,f_{\pi_{2,v}}, f_{\pi_{2,v}},[\rho(\varphi_{1,v}\otimes\bar{\varphi}_{1,v})])$ absolutely converge for $\Re(s)>-\frac{1}{2}$ and so $Z_v(0,f_{\mu_v}, f_{\mu_v},[\rho(\varphi_{2,v}\otimes\bar{\varphi}_{2,v})])$ does. For $\Re(s)>0$, $I_v(s)$ is just the product of $Z_v(s,f_{\pi_{2,v}}, f_{\pi_{2,v}},[\rho(\varphi_{1,v}\otimes\bar{\varphi}_{1,v})])$ and $Z_v(0,f_{\mu_v}, f_{\mu_v},[\rho(\varphi_{2,v}\otimes\bar{\varphi}_{2,v})])$, the above doubled integral for $I_v(s)$ absolutely converges for $\Re(s)>0$.\\



$Step\ 4.$ By making a change of variables $g_v\to g_vl_v$, \beqnan&&I_v(s,\varphi_{1,v}, \varphi_{2,v},f_{\pi_{2,v}},f_{\mu_v})=\int_{U(V)_v \times U(L)_v} J(s,g_vl_v,l_v,\varphi_{1,v}, \varphi_{2,v},f_{\pi_{2,v}},f_{\mu_v}) dg_v dl_v\\&&=\int_{U(V)_v \times U(L)_v} \mathcal{B}_{\omega_{ W,V}}^{\varphi_{1,v}}(g_vl_v) \cdot \mathcal{B}_{\omega_{W,L}}^{\varphi_{2,v}}(l_v)\cdot \mathcal{B}_{\pi_{2,v}}^{ f_{\pi_{2,v}}}(g_vl_v)\cdot  \mathcal{B}_{\mu_v}^{ f_{\mu_v}}(l_v)\cdot \Delta_2(g_vl_v)^{s}dg_v dl_v\\&&=\int_{U(V)_v \times U(L)_v} \mathcal{B}_{\omega_{W,V\oplus L}}^{\varphi_v}(g_v,l_v) \cdot \mathcal{B}_{\pi_{2,v}}^{ f_{\pi_{2,v}}}(g_v)\cdot \omega_{\pi_{2,v}}(l_v) \cdot \mathcal{B}_{\mu_v}^{ f_{\mu_v}}(l_v)\cdot \Delta_2(g_vl_v)^{s}dg_v dl_v\\&&=\int_{U(V)_v \times U(L)_v} \mathcal{B}_{\omega_{W,V\oplus L}}^{\varphi_v}(g_v,l_v) \cdot \mathcal{B}_{\pi_{2,v}}^{ f_{\pi_{2,v}}}(g_v)\cdot \omega_{\sigma_v}(l_v) \cdot \mathcal{B}_{\mu_v}^{ f_{\mu_v}}(1_v)\cdot \Delta_2(g_vl_v)^{s}dg_v dl_v\\&&=\int_{U(V)_v \times U(L)_v} \mathcal{B}_{\omega_{W,V\oplus L}}^{\varphi_v}(g_v,l_v) \cdot \mathcal{B}_{\pi_{2,v}}^{ f_{\pi_{2,v}}}(g_v)\cdot \mathcal{B}_{\sigma_v}^{ f_{\sigma_v}}(l_v)\cdot \Delta_2(g_vl_v)^{s}dg_v dl_v\eeqnan
(The last equality follows from $\mathcal{B}_{\sigma_v}(f_{\sigma_v},f_{\sigma_v})=\mathcal{B}_{\mu_v}(f_{\mu_v},f_{\mu_v}$)). \\

$Step\ 5.$ By the lemma \ref{lem2} in the next section, we see that $$\lim_{\Re(s)\to 0+} \frac{\zeta_v(2s)}{L_{E_v}(s, BC(\pi_{2,v})\otimes \gamma_v)}\cdot \int_{U(V)_v \times U(L)_v} \mathcal{B}_{\omega_{W,V\oplus L}}^{\varphi_v}(g_v,l_v) \cdot \mathcal{B}_{\pi_{2,v}}^{ f_{\pi_{2,v}}}(g_v)\cdot \mathcal{B}_{\sigma_v}^{f_{\sigma_v}}(l_v)\cdot \Delta_2(g_vl_v)^{s}dg_v dl_v=$$ 

$$\lim_{\Re(s)\to 0+} \frac{\zeta_v(2s)}{L_{E_v}(s, BC(\pi_{2,v})\otimes \gamma_v)}\cdot\int_{U(V)_v \times U(L)_v} \mathcal{B}_{\omega_{W,V\oplus L}}^{\varphi_v}(g_v,l_v) \cdot \mathcal{B}_{\pi_{2,v}}^{ f_{\pi_{2,v}}}(g_v)\cdot \mathcal{B}_{\sigma_v}^{ f_{\sigma_v}}(l_v)\cdot \Delta_2(g_v)^{s}dg_v dl_v=$$  $$\lim_{\Re(s)\to 0+}\frac{\zeta_v(2s)}{L_{E_v}(s, BC(\pi_{2,v})\otimes \gamma_v)} \cdot \int_{U(V)_v}Z_v(1, {f}_{\sigma_v}, {f}_{\sigma_v}, [\rho(g_v \cdot \varphi_{v}\otimes\bar{\varphi}_{v})])\cdot \mathcal{B}_{\pi_{2,v}}^{ f_{\pi_{2,v}}}(g_v)\cdot \Delta_2(g_v)^{s} dg_v.$$ 

\noindent We normalize $Z_v(1, {f}_{\sigma_v}, {f}_{\sigma_v}, [\rho(\varphi_{v}\otimes\bar{\varphi}_{v})])$ by $$Z_v^{\sharp}(1, {f}_{\sigma_v}, {f}_{\sigma_v}, [\rho(\varphi_{v}\otimes\bar{\varphi}_{v})]):=\frac{L_{v}(3,\chi_{E_v/F_v})}{L_{E_v}(3/2, BC(\sigma_v)\otimes \gamma_v^3)}\cdot Z_v(1, {f}_{\sigma_v}, {f}_{\sigma_v}, [\rho(\varphi_{v}\otimes\bar{\varphi}_{v})]).$$

We define the local inner product $\mathcal{B}_{\theta(\bar{\sigma}_v)}$ on $\theta_v(\bar{\sigma}_v)$ as follows:
$$\mathcal{B}_{\theta(\bar{\sigma}_v)}(\theta_v(\bar{f}_{\sigma_v},\varphi_{v}),\theta_v(\bar{f}_{\sigma_{v}},\varphi_{v})):=
\begin{cases}\frac{L_{E}(3/2, BC(\sigma)\otimes \gamma^3)}{L(3,\chi_{E/F})} \cdot Z_v^{\sharp}(1, {f}_{\sigma_v}, {f}_{\sigma_v}, [\rho(\varphi_{v}\otimes\bar{\varphi}_{v})]) \mbox{   for some place $v$} \\\\ Z_v^{\sharp}(1, {f}_{\sigma_v}, {f}_{\sigma_v}, [\rho(\varphi_{v}\otimes\bar{\varphi}_{v})]) \mbox{   for the remaining places}
\end{cases}$$ 

Then we see that $$\mathcal{B}_{\Theta(\bar{\sigma})}(\theta(\bar{f}_{\sigma},\varphi),\theta(\bar{f}_{\sigma},\varphi))=\prod_{v} \mathcal{B}_{\theta(\bar{\sigma}_v)}(\theta_v(\bar{f}_{\sigma_v},\varphi_{v}),\theta_v(\bar{f}_{\sigma_{v}},\varphi_{v}))$$
and $\mathcal{B}_{\theta(\bar{\sigma}_v)}(\theta_v(\bar{f}_{\sigma_v},\varphi_{v}),\theta_v(\bar{f}_{\sigma_{v}},\varphi_{v}))=1$ for unramified data $(f_{\sigma_v},\varphi_{v})$. \\(Note that the $'$small$'$ local theta-lift is the maximal semisimple quotient of the $'$big$'$ theta-lift,  and so we should check whether these pairings are well-defined. But since we are assuming $\Theta(\bar{\sigma})$ is cuspidal, it is semisimple and so $\mathcal{B}_{\Theta(\bar{\sigma})}(\theta(\bar{f}_{\sigma},\varphi),\theta(\bar{f}_{\sigma},\varphi))$ factors as a map $\sigma_v \otimes \bar{\sigma_v} \otimes \varpi_{\omega_{W,V\oplus L}} \otimes {\bar{\varpi}_{\omega_{W,V\oplus L}}} \to $ $\Theta(\bar{\sigma}) \otimes \overline{\Theta(\bar{\sigma})}$. Thus theorem (\ref{thm3.4}) shows that $\mathcal{B}_{\Theta(\bar{\sigma}_v)}$ descends to $\mathcal{B}_{\theta(\bar{\sigma}_v)}$.)\\

$Step\ 6.$ With the things we developed so far, we see that $$\mathcal{P}(f_{\Theta(\bar{\sigma})},f_{\pi_2})=\frac{1}{4}\mathcal{P}'(f_{\Theta(\bar\sigma)}, f_{\pi_2}, f_{\mu})=\frac{1}{4}\mathcal{P}'(\theta(\bar{f_{\pi_2}},\varphi_1)), \theta(\bar{f_{\mu}},\varphi_2)), f_{\sigma})$$
$$=-\frac{1}{2^3}\cdot \frac{L_E(\frac{1}{2}, BC(\mu)\otimes \gamma)}{L(1,\chi_{E/F})}\cdot \frac{Res_{s=0}(L_E(s, BC(\pi_2)\otimes \gamma))}{L(1,\chi_{E/F})} \cdot \prod_{v} Z_v^{\sharp}(f_{\mu_v},f_{\pi_{2,v}},\varphi_{1,v},\varphi_{2,v})$$
$$=-\frac{1}{2^3} \cdot \frac{L_E(\frac{1}{2}, BC(\mu)\otimes \gamma)}{L(1,\chi_{E/F})}\cdot \frac{Res_{s=0}(L_E(s, BC(\pi_2)\otimes \gamma))}{L(1,\chi_{E/F})} \frac{L(3,\chi_{E/F})}{L_{E}(3/2, BC(\sigma)\otimes \gamma^3)} \cdot \prod_{v} \mathcal{P}_v(\theta_v(\bar{f}_{\sigma_v},\varphi_{v}),f_{\pi_2,v})$$

\noindent This proves the theorem.

\begin{remark} In the course of the proof, we see that our local period $\mathcal{P}_v^{\sharp}(\theta_v(\bar{f}_{\sigma_v},\varphi_{v}),f_{\pi_2,v})$ is just the unfolding expression of $Z_v^\sharp(0, {f}_{\mu_v}, {f}_{\mu_v}, \Phi_{0, v})\cdot Z_{v,s=-\frac{1}{2}}^\sharp(s, {f}_{\pi_{2,v}}, {f}_{\pi_{2,v}}, \Phi_{-\frac{1}{2}, v})$. By Proposition 11.6 in \cite{Gan}, the non-vanishing of these two local zeta integral functions $Z_{v,s=-\frac{1}{2}}^\sharp$ and $Z_v^\sharp$ is equivalent to the non-vanishing of the local theta lifts $\Theta_{W_{1,v},V_{2,v}}(\overline{\pi_{2,v}})$ and $\Theta_{W_{1,v},L_{1,v}}(\overline{\mu_v})$ respectively. Since we assumed that theta lift $\Theta_{W_{1},V_{2}}(\overline{\pi_{2}})$ is non-zero, $\Theta_{W_{1,v},V_{2,v}}(\overline{\pi_{2,v}})$ is non-zero for all $v$ and so is  $Z_{v,s=-\frac{1}{2}}^\sharp$. Thus the non-vanishing of $\mathcal{P}_v^{\sharp}$ is equivalent to $\Theta_{W_{1,v},L_{1,v}}(\overline{\mu_v})=\Theta_{W_{1,v},L_{1,v}}(\overline{\sigma_v})$ is non-zero. (Note that $w_{\pi_{2}}=\mathbb{I}$.) Based upon this observation, we proved that $\mathcal{P}_v^{\sharp}\ne0$ is equivalent to $\text{Hom}_{U(V_2)(F_v)}(\pi_{3,v},\pi_{2,v})\ne0$ in \cite{Haan} and it is the analog of the local Ichino-Ikeda conejctures for non-tempered case.
\end{remark}

\section{Proof of Lemma \ref{lem2}}\label{last}

\noindent In this section, we prove the lemma upon which we developed $Step\ 5$ in the proof of \ref{thm}. We retain the same notations as in the previous section and since everything occurs in local case, we suppress $v$ from the notation. We remind the reader that $\pi_2$ is given by the theta lift of the trivial character $\mathbb{I}$ of $U(1)$.
\begin{lemma}Let $t$ be the order of $\frac{\zeta(2s)}{L_{E}(s, BC(\pi_{2})\otimes \gamma)}$ at $s=0$. Then,

\begin{equation}\label{lem2}\lim_{\Re(s)\to 0+}s^t \cdot \int_{U(V) \times U(L)} \mathcal{B}_{\omega_{W,V\oplus L}}^{\varphi}(g,l) \cdot \mathcal{B}_{\pi_2}^{ f_{\pi_2}}(g)\cdot \mathcal{B}_{\sigma}^{ f_{\sigma}}(l)\cdot (\Delta_2(gl)^{s}-\Delta_2(g)^{s})dg dl=0\end{equation}
\end{lemma}
\begin{proof}When $E$ is quadratic field extension of $F$, $U(L)$ is the centralizer of $U(V)$ and compact and so it is included in every maximal compact subgroup of $U(V)$. Then $\Delta_2(gl)^{s}-\Delta_2(g)^{s}=0$ and so the lemma is immediate in this case. So we assume $E=F \times F$ and by our hypothesis, all archimedean places do not split, and so we consider only $p$-adic case.

\noindent Since $E=F \times F$, $U(n)\simeq GL_n(F)$ and by Cartan decomposition, $GL_1(F)=\bigcup_{l\in \Z} \varpi^lK_{1}$, $GL_2(F)=\bigcup_{n\in \Z, m\in \N}K_{2} \begin{pmatrix} \varpi^{n+m} & \\  & \varpi^{n}\end{pmatrix} K_{2}$. (here, $\mathcal{O}$ is the ring of integer of $F$ and $\varpi$ is a uniformizer of $\mathcal{O}$ and $K_i=GL_i(\mathcal{O})$.) \\Since the theta lift preserves the central character, $\omega_{\pi_2}(\varpi)=1$ and let $\alpha=\sigma(\varpi)$. For $i=1,2$ and diagonal matrix $m\in GL_i(F),$ let $\mu_i(m):=\frac{Vol(K_imK_i)}{Vol(K_i)^2}$. Since $GL_1(F)$ is abelian, $\mu_1(m)=1$ and by the Lemma 2.1 in (\cite{Qiu2}), $\mu_2(\operatorname{diag}(a,b))=C\cdot |\frac{b}{a}|$ for some constant $C\in \mathbb{R}_{>0}.$

\noindent Then the measure decomposition formula turns \ref{lem2} to show

$$\lim_{\Re(s)\to 0+}s^t \cdot \sum_{n,l\in \Z , m\ge0}\alpha^l\cdot|\varpi|^{-m}\cdot  (|\varpi|^{s(|n+m+l|+|n+l|)}-|\varpi|^{s(|n+m|+|n|)}) \cdot I(s,\varphi,f_{\pi_2},m,n,l)=0$$ where $I(s,\varphi,f_{\pi_2},m,n,l)=$

$$\int_{K_1 \times K_2 \times K_2} \mathcal{B}_{\omega_{W,V\oplus L}}^{\varphi}(k_2\operatorname{diag}(\varpi^{n+m},\varpi^{n})k'_2,\varpi^lk_1) \cdot \mathcal{B}_{\pi_2}^{ f_{\pi_2}}(k_2\operatorname{diag}(\varpi^{m},1)k'_2)dk_1 dk_2 dk'_2.$$ Since $\varphi$ and $f_{\pi_2}$ are $K \times K$-finite functions, we are sufficient to show \beqnan \lim_{\Re(s)\to 0+} s^t \cdot \Big(\sum_{n,l\in \Z , m\ge0}\alpha^l\cdot|\varpi|^{-m}\cdot  (|\varpi|^{s(|n+m+l|+|n+l|)}-|\varpi|^{s(|n+m|+|n|)}) \cdot c_{n,m,l}\cdot d_{m}\Big)=0 \eeqnan  where $c_{n,m,l}=\mathcal{B}_{\omega_{W,V\oplus L}}^{\varphi}(\operatorname{diag}(\varpi^{n+m},\varpi^{n}),\varpi^l)$ and $d_{m}=\mathcal{B}_{\pi_2}^{ f_{\pi_2}}(\operatorname{diag}(\varpi^{m},1))$. \\Now we invoke the asymptotic fomulas of $c_{n,m,l}$ and $d_m$.
Recall (\ref{weil}) in Section 2.2 and write $c=\gamma_1^2(\varpi)$. (Note that $|c|=1$.) Since $\varphi$ is locally constant and has compact support, there is $l_1\in \N$ such that for $X,Y \in F^3$, if $|X-Y|\le |\varpi|^{l_1}\cdot \operatorname{Sup}\{|X| \mid X\in supp(\varphi) \subset F^3\}$, then $\varphi(X)=\varphi(Y)$. Thus

 $$c_{n,m,l}=\begin{cases}c^{2n+m+l}\cdot|\varpi|^{n+\frac{m+l}{2}}\cdot \int_{F^3} \varphi(\varpi^{n+m}x_1,\varpi^{n}x_2,0) \cdot \varphi(x_1,x_2,x_3) dX, \text{ if } l\ge l_1 \\c^{2n+m+l} \cdot|\varpi|^{n+\frac{m-l}{2}}\cdot \int_{F^3} \varphi(\varpi^{n+m}x_1,\varpi^{n}x_2,x_3) \cdot \varphi(x_1,x_2,0) dX, \text{ if } l\le -l_1.\end{cases}$$
Write $$a_{n,m}=\int_{F^3} \varphi(\varpi^{n+m}x_1,\varpi^{n}x_2,0) \cdot \varphi(x_1,x_2,x_3) dX, \ b_{n,m}=\int_{F^3} \varphi(\varpi^{n+m}x_1,\varpi^{n}x_2,x_3) \cdot \varphi(x_1,x_2,0) dX.$$ \\Then $a_{n,m}=\begin{cases} a^1_{n,m}, &\text{ if } n\ge
l_1 \\|\varpi|^{-n}\cdot a^2_{n,m}, &\text{ if } n\le -l_1, \end{cases}$ where $$a^1_{n,m}=\int_{F^3} \varphi(\varpi^{n+m}x_1,0,0) \cdot \varphi(x_1,x_2,x_3) dX, \ a^2_{n,m}=\int_{F^3} \varphi(\varpi^{n+m}x_1,x_2,0) \cdot \varphi(x_1,0,x_3) dX$$ and $b_{n,m}=\begin{cases}b^1_{n,m}, &\text{ if } n\ge
l_1 \\|\varpi|^{-n}\cdot b^2_{n,m}, &\text{ if } n\le -l_1, \end{cases}$ where 
 $$b^1_{n,m}=\int_{F^3} \varphi(\varpi^{n+m}x_1,0,x_3) \cdot \varphi(x_1,x_2,0) dX,\ b^2_{n,m}=\int_{F^3} \varphi(\varpi^{n+m}x_1,x_2,x_3) \cdot \varphi(x_1,0,0) dX.$$ Again $$a^i_{n,m}=\begin{cases} k^i_1 &\text{ if } n+m\ge l_1\\ |\varpi|^{-(n+m)}\cdot k^i_2 &\text{ if } n+m\le -l_1 \end{cases}$$ and $$b^i_{n,m}=\begin{cases}k^i_3 &\text{ if } n+m\ge l_1\\ |\varpi|^{-(n+m)}\cdot k^i_4 &\text{ if } n+m\le -l_1 \end{cases}$$ for some constants $\{k^i_1,k^i_2,k^i_3,k^i_4 \}_{i=1,2}.$\\

\noindent Note that in codimension $0,1$ case, the theta lift sends a tempered representation to a tempered one. Thus we know that $\pi_2$ is tempered and by [Prop.8.1, \cite{Gan1}], we see that it is the irreducible unitary induced representation $B(\gamma_1^2,\gamma_1^{-2})$ of $GL(2)(F)$. (here, since $\gamma=(\gamma_1,\gamma_1^{-1})$, if we regard $\gamma$ as a character of $F^{\times}$ using the isomorphism of $U(1)$ and $GL(1)$, $\gamma(x)=\gamma_1^2(x)$.) Then by (\cite{Qiu2}, Lemma 3.9), if we take $l_1$ large enough, we assume that for $m\ge l_1$, $d_m=|\varpi|^{\frac{m}{2}}\cdot(c_1\cdot c^m + c_2 \cdot c^{-m})$  where $c_1,c_2$ are constants.

If $\pi$ is an unramified representation of $U(W_n)$ and $\theta(\pi)$ is the theta lift of $\pi$ to $U(V_{n+1})$, then $BC(\theta(\pi))\simeq BC(\pi)\gamma^{-1} \boxplus \gamma^{n}$ by (8.1.2) in \cite{Xue}. Recall that $GU_{2,0}(\A_F) \simeq (D^{\times} \times E^{\times})/\Delta F^{\times}$ where $D$ is the quaternion division algebra over $F$ and $GU_{1,1}(\A_F)\simeq (GL_2(F) \times E^{\times})/\Delta F^{\times}$. Since $GL_2(F)$ and $D^{\times}$ have the strong multiplicity one theorem and global theta lift is the product of local theta lifts, the unramified computations of the local theta lifts completely determine the global theta lift from $U(1)$ to $U(2)$ not at the level of individual represenations but of $L$-parameters. Thus since $\pi_2$ is the theta lift of the trivial representation, we have the $L$-parameter relation $BC(\pi_2)=BC(\mathbb{I})\gamma^{-1} \boxplus \gamma$ for all places and so $L_{E}(s,BC(\pi_2)\otimes \gamma)=(\frac{1}{1-q^{-s}})^2 \cdot \frac{1}{1-\gamma^2_1(\varpi)q^{-s}}\cdot \frac{1}{1-\gamma_1^{-2}(\varpi)q^{-s}}$. (Recall $\gamma=(\gamma_1,\gamma_1^{-1}$) for some unitary character $\gamma_1$ of $F$.) Thus if $\gamma_1^2(\varpi)=1$, $L_{E}(s,BC(\pi_2)\otimes \gamma)$ has a quadruple pole at $s=0$ and if $\gamma_1^2(\varpi)\ne 1$, then it has double pole at $s=0$.) So in any cases, we have $t\ge1$.\\

\noindent Now, we introduce two notation that we will use in this argument :\\

\noindent $\bullet$ If two meromorphic functions $f_1,f_2$ differ by a constant multiplication, we write $f_1 \approx f_2$.\\

\noindent $\bullet$ For two meromorphic functions $f_1,f_2$ and $m \in \N$, if $\lim_{\Re(s) \to 0+} s^m \cdot (f_1(s)-f_2(s))=0$, we write $f_1 \overset{m}{\sim} f_2$ and if $f_1 \overset{0}{\sim} f_2$, we simply write $f_1 \sim f_2$. \\

\noindent Since the integral in the Lemma absolutely converges on $\Re(s)>0$, to prove it, it suffices to show that each component of the integral \begin{equation} \label{1} \sum_{n\in \Z , m\ge0}c^{2n+m}|\varpi|^{n-\frac{m}{2}}d_ma_{n,m}\cdot \big(\sum_{l\ge l_1}c^l\alpha^l|\varpi|^{\frac{l}{2}}  (|\varpi|^{s(|n+m+l|+|n+l|)}-|\varpi|^{s(|n+m|+|n|)}\big)\end{equation} \\  \begin{equation} \label{2}  \sum_{n\in \Z , m\ge0}|\varpi|^{-m}d_m\cdot \big(\sum_{-l_1<l<l_1}\alpha^l (|\varpi|^{s(|n+m+l|+|n+l|)}-|\varpi|^{s(|n+m|+|n|)}) \cdot c_{n,m,l}\big)\end{equation}\\ \begin{equation} \label{3} \sum_{n\in \Z , m\ge0}c^{2n+m}|\varpi|^{n-\frac{m}{2}}d_ma_{n,m}\cdot \big(\sum_{l<- l_1}c^l\alpha^l|\varpi|^{\frac{l}{2}}  (|\varpi|^{s(|n+m+l|+|n+l|)}-|\varpi|^{s(|n+m|+|n|)}\big)  \end{equation} \\ are all $\overset{1}{\sim} 0.$\\

\noindent We will first show (\ref{1})$\overset{1}{\sim} 0.$ To do this, we write $$r_{l,m,n}(s)=c^l\alpha^l|\varpi|^{\frac{l}{2}}  (|\varpi|^{s(|n+m+l|+|n+l|)}-|\varpi|^{s(|n+m|+|n|)}\big)$$ and decompose (\ref{1}) into three component.$$\sum_{n\in \Z , m\ge0}c^{2n+m}|\varpi|^{n-\frac{m}{2}}d_ma_{n,m}\cdot \Big(\sum_{l\ge l_1, l < -(n+m)}r_{l,m,n}(s)\Big)$$ $$+\sum_{n\in \Z , m\ge0}c^{2n+m}|\varpi|^{n-\frac{m}{2}}d_ma_{n,m}\cdot \Big(\sum_{l\ge l_1, -(n+m)\le l < -n} r_{l,m,n}(s)\Big)$$ $$+\sum_{n\in \Z , m\ge0}c^{2n+m}|\varpi|^{n-\frac{m}{2}}d_ma_{n,m}\cdot \Big(\sum_{l\ge l_1, l\ge -n} r_{l,m,n}(s)\Big)$$ and show each component is $\overset{1}{\sim} 0.$

\noindent For fixed $m\in \N$ and small $\Re(s)>0$, $$ \sum_{n\in \Z}c^{2n+m}|\varpi|^{n-\frac{m}{2}}d_ma_{n,m}\cdot \sum_{l\ge l_1, l < -(n+m)}r_{l,m,n}(s)=$$ $$ \sum_{n\le -(m+l_1+1)}c^{2n+m}\cdot |\varpi|^{n(1-2s)-m(\frac{1}{2}+s)}d_ma_{n,m}\cdot \Big(f_1(s)-f^{m,n}_2(s)\Big) $$ where $$f_1(s)=\frac{(c\alpha|\varpi|^{\frac{1}{2}-2s})^{l_1}}{1-c\alpha|\varpi|^{\frac{1}{2}-2s}}-\frac{(c\alpha|\varpi|^{\frac{1}{2}})^{l_1}}{1-c\alpha|\varpi|^{\frac{1}{2}}} \text{ and }$$ $$ f^{m,n}_2(s)=\frac{(c\alpha|\varpi|^{\frac{1}{2}-2s})^{-(n+m)}}{1-c\alpha|\varpi|^{\frac{1}{2}-2s}}-\frac{(c\alpha|\varpi|^{\frac{1}{2}})^{-(n+m)}}{1-c\alpha|\varpi|^{\frac{1}{2}}}.$$
Note that $f_1, f^{m,n}_2 \sim 0.$

\noindent Since $$\sum_{m\ge0}\sum_{n\le -(m+l_1+1)}c^{2n+m}\cdot|\varpi|^{n(1-2s)-m(\frac{1}{2}+s)}d_ma_{n,m}$$ $$\approx \sum_{m\ge0}d_m(c|\varpi|^{-s-\frac{3}{2}})^m\Big(\sum_{n\le -(m+l_1+1)}(c^2 |\varpi|^{-1-2s})^n\Big)$$ $$\approx (c^{-2}|\varpi|^{1+2s})^{l_1+1} \sum_{m\ge0} \frac{d_m(c^{-1}|\varpi|^{s-\frac{1}{2}})^m}{1-c^{-2}|\varpi|^{1+2s}}$$ $$=\frac{(c^{-2}|\varpi|^{1+2s})^{l_1+1}}{1-c^{-2}|\varpi|^{1+2s}}\cdot  \Big(\big(\sum_{m=0}^{l_1-1} d_m(c^{-1}|\varpi|^{s-\frac{1}{2}})^m\big)+ c_1 \cdot \frac{|\varpi|^{sl_1}}{1-|\varpi|^s}+ c_2 \cdot \frac{(c^{-2}|\varpi|^s)^{l_1}}{1-c^{-2}|\varpi|^s}\Big)$$ and so $( \sum_{m\ge0}\sum_{n\le -(m+l_1+1)}c^{2n+m}\cdot|\varpi|^{n(1-2s)-m(\frac{1}{2}+s)}d_ma_{n,m})\cdot f_1(s) \overset{1}{\sim} 0$. \\

\noindent Furthermore, $$ \sum_{m\ge0}\sum_{n\le -(m+l_1+1)}c^{2n+m}\cdot |\varpi|^{n(1-2s)-m(\frac{1}{2}+s)}d_ma_{n,m}\cdot \Big(\frac{(c\alpha|\varpi|^{\frac{1}{2}-2s})^{-(n+m)}}{1-c\alpha|\varpi|^{\frac{1}{2}-2s}}-\frac{(c\alpha|\varpi|^{\frac{1}{2}})^{-(n+m)}}{1-c\alpha|\varpi|^{\frac{1}{2}}}\Big)\approx$$ $$\sum_{m\ge0}d_m \alpha^{-m}\Big( \sum_{n \le -(m+l_1+1)}|\varpi|^{(s-2)m}\cdot (c|\varpi|^{-\frac{3}{2}}\alpha^{-1})^n + \ |\varpi|^{(-s-2)m}\cdot(c|\varpi|^{-\frac{3}{2}-2s}\alpha^{-1})^n \Big)=$$ $$\big(\frac{(c^{-1}|\varpi|^{\frac{3}{2}}\alpha)^{l_1+1}}{1-c^{-1}|\varpi|^{\frac{3}{2}}\alpha}-\frac{(c^{-1}|\varpi|^{\frac{3}{2}}\alpha)^{l_1+1}}{1-c^{-1}|\varpi|^{\frac{3}{2}+2s}\alpha}\big)\cdot \sum_{m\ge0}d_m(c^{-1}|\varpi|^{(s-\frac{1}{2})})^m\approx$$ $$\big(\frac{(c^{-1}|\varpi|^{\frac{3}{2}}\alpha)^{l_1+1}}{1-c^{-1}|\varpi|^{\frac{3}{2}}\alpha}-\frac{(c^{-1}|\varpi|^{\frac{3}{2}}\alpha)^{l_1+1}}{1-c^{-1}|\varpi|^{\frac{3}{2}+2s}\alpha}\big)\cdot \Big( c_1 \cdot \frac{|\varpi|^{sl_1}}{1-|\varpi|^s}+ c_2 \cdot \frac{(c^{-2}|\varpi|^s)^{l_1}}{1-c^{-2}|\varpi|^s}\Big) \overset{1}{\sim} 0.$$ Thus we see that $$\sum_{n\in \Z , m\ge0}c^{2n+m}|\varpi|^{n-\frac{m}{2}}d_ma_{n,m}\cdot \big(\sum_{l\ge l_1, l < -(n+m)}r_{l,m,n}(s) \big)\overset{1}{\sim} 0.$$\\

\noindent Next we will show $$\sum_{m \in \N}c^md_m |\varpi|^{-\frac{m}{2}}\sum_{n\in \Z}c^{2n}|\varpi|^n a_{n,m}\cdot \sum_{l\ge l_1, -(n+m)\le l < -n} (c\alpha|\varpi|^{\frac{1}{2}})^l \cdot  (|\varpi|^{sm}-|\varpi|^{s(|n+m|-n)})\overset{1}{\sim} 0.$$ Let $$p_{n,m}(s)=c^{2n}|\varpi|^n a_{n,m}\cdot \sum_{l\ge l_1, -(n+m)\le l < -n} (c\alpha|\varpi|^{\frac{1}{2}})^l \cdot  (|\varpi|^{sm}-|\varpi|^{s(|n+m|-n)}).$$ Then $$\sum_{n \in \Z}p_{n,m}(s)=$$ $$\sum_{n<\operatorname{min}\{-l_1,-m\}}c^{2n}|\varpi|^n a_{n,m}\cdot(|\varpi|^{sm}-|\varpi|^{(-2n-m)s})\cdot \frac{(c\alpha|\varpi|^{\frac{1}{2}})^{\operatorname{max}\{l_1,-(n+m)\}}-(c\alpha|\varpi|^{\frac{1}{2}})^{-n}}{1-c\alpha|\varpi|^{\frac{1}{2}}}$$ and so to show $\sum_{m \in \N}c^md_m |\varpi|^{-\frac{m}{2}}\sum_{n\in \Z}p_{n,m}(s) \overset{1}{\sim}0$, it is suffcient to check \beqna \label{4} \sum_{0\le m < l_1}c^md_m |\varpi|^{-\frac{m}{2}}\cdot \big(\sum_{-l_1-m < n < -l_1}p_{n,m}(s)\big) \overset{1}{\sim}0 &&\\ \label{5} \sum_{0\le m < l_1}c^md_m |\varpi|^{-\frac{m}{2}}\cdot \big(\sum_{n\le -l_1-m}p_{n,m}(s) \big)\overset{1}{\sim}0 &&\\ \label{6} \sum_{m \ge l_1}c^md_m |\varpi|^{-\frac{m}{2}}\cdot \big(\sum_{-l_1-m < n \le -m}p_{n,m}(s)\big) \overset{1}{\sim}0 &&\\ \label{7} \sum_{m \ge l_1}c^md_m |\varpi|^{-\frac{m}{2}}\cdot \big(\sum_{n\le -l_1-m}p_{n,m}(s)\big) \overset{1}{\sim}0 . \eeqna
\noindent For each $0\le m<l_1, -l_1-m\le n<-l_1$, $$c^md_m|\varpi|^{-\frac{m}{2}}p_{n,m}(s)\overset{1}{\sim}0 $$ and so (\ref{4}) easily follows.\\
\noindent For each $m \in \N$, $$\sum_{n \le -l_1-m}p_{n,m}(s)\approx (c^{-2}|\varpi|^s)^m \cdot g_1(s) - (c^{-1}\alpha |\varpi|^{\frac{1}{2}+s})^m \cdot g_2(s) $$ where $$g_1(s)=\frac{(c^{-1}\alpha|\varpi|^{\frac{3}{2}})^{l_1}}{1-c^{-2}|\varpi|}-\frac{(c^{-1}\alpha|\varpi|^{\frac{3}{2}+2s})^{l_1}}{1-c^{-2}|\varpi|^{1+2s}}, \text{ }g_2(s)=\frac{(c^{-1}\alpha |\varpi|^{\frac{3}{2}})^{l_1}}{1-c^{-1}\alpha|\varpi|^{\frac{3}{2}}}-\frac{(c^{-1}\alpha |\varpi|^{\frac{3}{2}+2s})^{l_1}}{1-c^{-1}\alpha|\varpi|^{\frac{3}{2}+2s}}$$and so (\ref{5}) and (\ref{7}) follow from this.\\

\noindent For each $-l_1 < k< 0$, note that $$\sum_{m \ge l_1}c^md_m|\varpi|^{-\frac{m}{2}}p_{k-m,m}(s)\approx$$ $$(1-|\varpi|^{-2ks})\cdot \sum_{m\ge l_1}(c_1|\varpi|^{sm}+c_2(c^{-2}|\varpi|^s)^{2m})\cdot \big((c\alpha|\varpi|^{\frac{1}{2}})^{l_1}-(c\alpha|\varpi|^{\frac{1}{2}})^{m-k}\big)\sim0$$ and so we have (\ref{6}).\\

\noindent To prove $$\sum_{m \in \N}c^md_m |\varpi|^{-\frac{m}{2}}\sum_{n\in \Z}c^{2n}|\varpi|^n a_{n,m}\sum_{l\ge l_1, l\ge-n}r_{l,m,n}(s) \overset{1}{\sim} 0, $$ we first decompose $\sum_{n\in \Z}c^{2n}|\varpi|^n a_{n,m}\sum_{l\ge l_1, l\ge-n}r_{l,m,n}(s) $ for fixed $m$ into three components $$\sum_{n\ge0}c^{2n}|\varpi|^n a_{n,m}\cdot |\varpi|^{s(2n+m)} ( \frac{(c\alpha|\varpi|^{\frac{1}{2}+2s})^{l_1}}{1-c\alpha|\varpi|^{\frac{1}{2}+2s}}- \frac{(c\alpha|\varpi|^{\frac{1}{2}})^{l_1}}{1-c\alpha|\varpi|^{\frac{1}{2}}}\big)$$ $+$
$$\sum_{-m\le n<0}c^{2n}|\varpi|^n a_{n,m}\big(|\varpi|^{s(2n+m)}\cdot \frac{(c\alpha|\varpi|^{\frac{1}{2}+2s})^{\operatorname{max}\{l_1,-n\}}}{1-c\alpha|\varpi|^{\frac{1}{2}+2s}}-|\varpi|^{sm}\cdot \frac{(c\alpha|\varpi|^{\frac{1}{2}})^{\operatorname{max}\{l_1,-n\}}}{1-c\alpha|\varpi|^{\frac{1}{2}}}\big)$$ $+$
$$\sum_{n < -m}c^{2n}|\varpi|^n a_{n,m}\big(|\varpi|^{s(2n+m)}\cdot \frac{(c\alpha|\varpi|^{\frac{1}{2}+2s})^{\operatorname{max}\{l_1,-n\}}}{1-c\alpha|\varpi|^{\frac{1}{2}+2s}}-|\varpi|^{-s(2n+m)}\cdot \frac{(c\alpha|\varpi|^{\frac{1}{2}})^{\operatorname{max}\{l_1,-n\}}}{1-c\alpha|\varpi|^{\frac{1}{2}}}\big).$$



\noindent Using the asymptotic formulae of $d_m$ and $a_{n,m},$ one can easily see that $$\sum_{m \in \N}c^md_m |\varpi|^{-\frac{m}{2}}\sum_{n\ge0}c^{2n}|\varpi|^n a_{n,m}\cdot |\varpi|^{s(2n+m)} ( \frac{(c\alpha|\varpi|^{\frac{1}{2}+2s})^{l_1}}{1-c\alpha|\varpi|^{\frac{1}{2}+2s}}- \frac{(c\alpha|\varpi|^{\frac{1}{2}})^{l_1}}{1-c\alpha|\varpi|^{\frac{1}{2}}}\big) \overset{1}{\sim} 0.$$\\ \noindent Write $p_{n,m}^1(s)=$ $$c^md_m |\varpi|^{-\frac{m}{2}}c^{2n}|\varpi|^n a_{n,m}\big(|\varpi|^{s(2n+m)}\cdot \frac{(c\alpha|\varpi|^{\frac{1}{2}+2s})^{\operatorname{max}\{l_1,-n\}}}{1-c\alpha|\varpi|^{\frac{1}{2}+2s}}-|\varpi|^{sm}\cdot \frac{(c\alpha|\varpi|^{\frac{1}{2}})^{\operatorname{max}\{l_1,-n\}}}{1-c\alpha|\varpi|^{\frac{1}{2}}}\big)$$ and note that $p_{n,m}^1(s) \sim 0.$ The second sum is decomposed into $$\sum_{0\le m <l_1}\sum_{-m\le n < 0}p_{n,m}^1(s)+\sum_{l_1 \le m}\sum_{-l_1\le n < 0}p_{n,m}^1(s)+\sum_{l_1 \le m}\sum_{-m\le n < -l_1}p_{n,m}^1(s)$$ and since $\sum_{0\le m <l_1}\sum_{-m\le n < 0}p_{n,m}^1(s)$ is a finite sum, it is $\overset{1}{\sim} 0.$ For each $-l_1 \le n <0$, one can easily check $\sum_{l_1\le m}p_{n,m}^1(s) \overset{1}{\sim} 0$ and so $\sum_{l_1 \le m}\sum_{-l_1\le n < 0}p_{n,m}^1(s) \overset{1}{\sim} 0.$

\noindent If $n<-l_1,$ $$|\varpi|^{s(2n+m)}\cdot \frac{(c\alpha|\varpi|^{\frac{1}{2}+2s})^{\operatorname{max}\{l_1,-n\}}}{1-c\alpha|\varpi|^{\frac{1}{2}+2s}}-|\varpi|^{sm}\cdot \frac{(c\alpha|\varpi|^{\frac{1}{2}})^{\operatorname{max}\{l_1,-n\}}}{1-c\alpha|\varpi|^{\frac{1}{2}}}=0$$ and so $\sum_{l_1 \le m}\sum_{-m\le n < -l_1}p_{n,m}^1(s) =0$. Thus the second sum $\sum_{m \in \N} \sum_{-m<n \le 0}p_{n,m}^1(s)=0.$\\

\noindent To show the third sum is $\overset{1}{\sim}0$, write $p_{n,m}^2(s)=$ $$c^{m+2n}d_m |\varpi|^{n-\frac{m}{2}} a_{n,m}\big(|\varpi|^{s(2n+m)}\cdot \frac{(c\alpha|\varpi|^{\frac{1}{2}+2s})^{\operatorname{max}\{l_1,-n\}}}{1-c\alpha|\varpi|^{\frac{1}{2}+2s}}-|\varpi|^{-s(2n+m)}\cdot \frac{(c\alpha|\varpi|^{\frac{1}{2}})^{\operatorname{max}\{l_1,-n\}}}{1-c\alpha|\varpi|^{\frac{1}{2}}}\big).$$\\

\noindent We decompose $$\sum_{m \in \N}\sum_{n<-m}p_{n,m}^2(s)=\sum_{m \in \N}\sum_{-m-l_1<n<-m}p_{n,m}^2(s)+\sum_{m \in \N}\sum_{n\le-m-l_1}p_{n,m}^2(s).$$ Write $k=m+n$ and for each $-l_1<k<0$, $$\sum_{m \in \N} p_{k-m,m}^2(s) \approx \sum_{m \ge l_1} p_{k-m,m}^2(s)=c^k(c_1(c\alpha|\varpi|^s)^m+c_2(c^{-1}\alpha|\varpi|^s)^m)\cdot g_k(s) \overset{1}{\sim} 0$$ where $$g_k(s)=\frac{(c\alpha|\varpi|^{\frac{1}{2}})^{-k}}{1-c\alpha|\varpi|^{\frac{1}{2}+2s}}-\frac{(c\alpha|\varpi|^{\frac{1}{2}+s})^{-k}}{1-c\alpha|\varpi|^{\frac{1}{2}}}.$$

\noindent Thus $\sum_{m \in \N}\sum_{-m-l_1<n<-m}p_{n,m}^2(s)=0.$\\

\noindent Next, for each $m \in \N$, some calculation shows that $$\sum_{n\le-m-l_1}p_{n,m}^2(s)=c^md_m|\varpi|^{-\frac{m}{2}}\cdot k_2^2 \cdot (c^{-1}\alpha |\varpi|^{\frac{1}{2}+s})^m \cdot g(s) \text{ where }$$ $$g(s)=\frac{(c^{-1}\alpha |\varpi|^{\frac{3}{2}})^{l_1}}{(1-c\alpha |\varpi|^{\frac{1}{2}+2s})(1-c^{-1}\alpha|\varpi|^{\frac{3}{2}})} - \frac{(c^{-1}\alpha |\varpi|^{\frac{3}{2}+2s})^{l_1}}{(1-c\alpha |\varpi|^{\frac{1}{2}})(1-c^{-1}\alpha|\varpi|^{\frac{3}{2}+2s})}$$ and so $\sum_{m \in \N}\sum_{n\le-m-l_1}p_{n,m}^2(s) \sim 0.$
Thus we have showed (\ref{1})$\overset{1}{\sim}0$.\\

\noindent Now, we will show (\ref{2}) $\overset{1}{\sim}0$.
To do this, for each $-l_1<l<l_1$, we decompose $$\sum_{n\in \Z , m\ge0}|\varpi|^{-m}d_m\cdot \alpha^l (|\varpi|^{s(|n+m+l|+|n+l|)}-|\varpi|^{s(|n+m|+|n|)}) \cdot c_{n,m,l} $$ into three summands $\sum_{m \in \N , n\ge l_1 } + \sum_{m \in \N , -l_1<n< l_1 } +\sum_{m \in \N , n\le -l_1 }$ and show that each is $\overset{1}{\sim}0$. \\\\Write $f_{n,m,l}(s)=|\varpi|^{-m}d_m\cdot \alpha^l (|\varpi|^{s(|n+m+l|+|n+l|)}-|\varpi|^{s(|n+m|+|n|)}) \cdot c_{n,m,l}$ and  note that for each fixed $n,m,l$, \ $f_{n,m,l} \sim 0.$ \\For each $-l_1<l<l_1$, we see that

\noindent $$\sum_{m \in \N , n\ge l_1}f_{n,m,l}(s)\approx \big(\sum_{n\ge l_1}(c^2 |\varpi|^{1+2s})^n\big)\cdot (\sum_{m \in \N}c_1 \cdot (c^2|\varpi|^s)^m+c_2 |\varpi|^{sm})\cdot (|\varpi|^{2ls}-1)\overset{1}{\sim} 0.$$

\noindent For all $-l_1<n,l<l_1$, there exists $N_1 \in \N$ such that $N_1>2l_1$ and if $m \ge N_1$, then $c_{n,m,l} = (c|\varpi|^{\frac{1}{2}})^m \cdot f_{n,l}$ for some constants $f_{n,l}.$ Thus $$\sum_{m\ge0, -l_1<n<l_1}f_{n,m,l}(s) =\sum_{0\le m<N_1, -l_1<n<l_1}f_{n,m,l}(s)+$$ $$\sum_{m \ge N_1, -l_1<n<l_1}c_1 (c^2|\varpi|^s)^m+c_2 |\varpi|^{sm}\big)\cdot(|\varpi|^{s(n+l+|n+l|)}-|\varpi|^{s(n+|n|)})\alpha^{l}\cdot f_{n,l}$$ and so $$\sum_{m\ge0, -l_1<n<l_1}f_{n,m,l}(s) \overset{1}{\sim}0.$$

\noindent Next we decompose $$\sum_{m\ge0, n \le- l_1}f_{n,m,l}$$ into four summands $$\sum_{n \le- l_1,m+n\ge \operatorname{max}\{-l,0\}}+\sum_{n \le- l_1,-l\le m+n<0}+\sum_{n \le- l_1,0\le m+n<-l}+\sum_{n \le- l_1,m+n< \operatorname{min}\{-l,0\}}.$$ \\
\noindent The first sum is zero. The second sum is $\sum_{-l\le k <0}\sum_{m \ge k+l_1}f_{k-m,m,l}$ and for each $-l\le k <0$, there exists $N_2\in \N$ such that $N_2\ge l_1$ and if $m \ge N_2$, then $c_{k-m,m,l}\approx |\varpi|^{\frac{m}{2}}\cdot c^{-m}$. Thus $$\sum_{-l\le k <0}\sum_{m \ge k+l_1}f_{k-m,m,l}\approx$$ $$(\sum_{k+l_1\le m < N_2}f_{k-m,m,l}) + \Big((1-|\varpi|^{-2ks})\cdot \sum_{m \ge N_2}(c_1 |\varpi^s|^{m}+c_2|c^{-2}\varpi^{s}|^{m})\Big)\overset{1}{\sim}0.$$ Similarly, we can show the third sum $\overset{1}{\sim}0.$\\
\noindent The fourth sum is decomposed into $$\sum_{n \le- l_1,-l_1<m+n< \operatorname{min}\{-l,0\}}f_{n,m,l}+\sum_{n \le- l_1,m+n\le-l_1}f_{n,m,l}$$ and as we have done in the above, it is easy to see $$\sum_{n \le- l_1,-l_1<m+n< \operatorname{min}\{-l,0\}}f_{n,m,l} \overset{1}{\sim}0.$$ Note
$$\sum_{n \le- l_1,m+n\le-l_1}f_{n,m,l}=\sum_{0\le m <l_1,n \le- l_1,m+n\le-l_1}f_{n,m,l}+\sum_{m\ge l_1,n \le- l_1,m+n\le-l_1}f_{n,m,l}.$$ For each $0\le m <l_1,$ $$\sum_{n\le-l_1-m}f_{n,m,l}\approx d_m(c|\varpi|^{-(\frac{3}{2}+s)})^m\cdot (|\varpi|^{-2ls}-1)\cdot \sum_{n\le-l_1-m}(c^2|\varpi|^{-(1+2s)})^n\overset{1}{\sim}0.$$ On the other hand,
$$\sum_{m\ge l_1}\sum_{n\le-l_1-m}f_{n,m,l}$$ $$\approx(|\varpi|^{-2ls}-1)(c_1(c^2|\varpi|^{-(1+s)})^m+c_2|\varpi|^{-(1+s)m})\cdot \sum_{n\le-l_1-m}(c^2|\varpi|^{-(1+2s)})^n$$ $$=(c^{-2}|\varpi|^{1+2s})^{l_1}\cdot (|\varpi|^{-2ls}-1)\cdot \big(\sum_{m\ge l_1}c_1\cdot|\varpi|^{sm}+c_2 \cdot (c^{-2}|\varpi|^{s})^m\big)\overset{1}{\sim} 0.$$\\
\noindent Thus we see that the fourth sum $\sum_{n \le- l_1,m+n< \operatorname{min}\{-l,0\}}f_{n,m,l}$ is also $\overset{1}{\sim} 0$ and we showed $(\ref{2}) \overset{1}{\sim} 0.$\\

\noindent Last, we will show (\ref{3}) $ \overset{1}{\sim} 0$.
\noindent To do this, write $\sum_{l\le -l_1}c^l  \alpha^l |\varpi|^{-\frac{l}{2}}  (|\varpi|^{s(|n+m+l|+|n+l|)}-|\varpi|^{s(|n+m|+|n|)}$ as $$\sum_{l\ge l_1}(c^{-1}\alpha^{-1}|\varpi|^{\frac{1}{2}})^{l} \cdot (|\varpi|^{s(|n+m-l|+|n-l|)}-|\varpi|^{s(|n+m|+|n|)}$$ and decompose it into three summands $$\sum_{l\ge l_1, l >(n+m)}\big( (c^{-1}\alpha^{-1}|\varpi|^{\frac{1}{2}+2s})^{l} \cdot  |\varpi|^{-s(2n+m)}-(c^{-1}\alpha^{-1}|\varpi|^{\frac{1}{2}})^{l}  \cdot |\varpi|^{s(|n+m|+|n|)}\big)$$ + $$\sum_{l\ge l_1, n< l \le n+m}(c^{-1}\alpha^{-1}|\varpi|^{\frac{1}{2}})^{l} \cdot  (|\varpi|^{sm}-|\varpi|^{s(|n+m|+|n|)})$$ + $$\sum_{l\ge l_1, l\le n}\big( (c^{-1}\alpha^{-1}|\varpi|^{\frac{1}{2}-2s})^{l}\cdot  |\varpi|^{s(2n+m)}-(c^{-1}\alpha^{-1}|\varpi|^{\frac{1}{2}})^{l}\cdot|\varpi|^{s(2n+m)}\big).$$

\noindent We write $M_{n,m}=\operatorname{max}\{l_1,m+n+1\}$. Then for fixed $m,n\in \N$ and small $\Re(s)>0$, $$\sum_{l\ge l_1, l >(n+m)}\big( (c^{-1}\alpha^{-1}|\varpi|^{\frac{1}{2}+2s})^{l} \cdot  |\varpi|^{-s(2n+m)}-(c^{-1}\alpha^{-1}|\varpi|^{\frac{1}{2}})^{l}  \cdot |\varpi|^{s(|n+m|+|n|)}\big)=$$ $$\frac{|\varpi|^{-s(2n+m)} (c^{-1}\alpha^{-1}|\varpi|^{\frac{1}{2}+2s})^{M_{n,m}} }{1-c^{-1}\alpha^{-1}|\varpi|^{\frac{1}{2}+2s} }-\frac{|\varpi|^{s(|n+m|+|n|)} (c^{-1}\alpha^{-1}|\varpi|^{\frac{1}{2}})^{M_{n,m}} }{1-c^{-1}\alpha^{-1}|\varpi|^{\frac{1}{2}}}.$$ Denote $$c^{2n+m}|\varpi|^{n-\frac{m}{2}}d_ma_{n,m} \cdot \Big(\frac{|\varpi|^{-s(2n+m)} (c^{-1}\alpha^{-1}|\varpi|^{\frac{1}{2}+2s})^{M_{n,m}} }{1-c^{-1}\alpha^{-1}|\varpi|^{\frac{1}{2}+2s} }-\frac{|\varpi|^{s(|n+m|+|n|)} (c^{-1}\alpha^{-1}|\varpi|^{\frac{1}{2}})^{M_{n,m}} }{1-c^{-1}\alpha^{-1}|\varpi|^{\frac{1}{2}}}\Big)$$ by $g_{n,m}(s)$ and note $g_{n,m}(s) \sim 0.$ We shall show $\sum_{m\ge0, n\in \N}g_{n,m}(s) \overset{1}{\sim}0.$ \\

\noindent Decompose $\sum_{m\ge0, n\in \N}g_{n,m}(s)$ into $$\sum_{m\ge0, n\ge0}g_{n,m}(s)+\sum_{m\ge-n, n<0}g_{n,m}(s)+\sum_{m<-n, n<0}g_{n,m}(s)$$ and the first sum decomposes again into $$\sum_{0\le m\le l_1-1}\sum_{0\le n \le l_1-m-1}g_{n,m}(s) +\sum_{0\le m\le l_1-1}\sum_{l_1-m\le n}g_{n,m}(s)+\sum_{l_1\le m}\sum_{0\le n }g_{n,m}(s).$$

\noindent Since the first term in the above is a finite sum, $\sum_{0\le m\le l_1-1}\sum_{0\le n \le l_1-m-1}g_{n,m}(s)\sim 0.$\\
For each $0\le m \le l_1-1$, $\sum_{l_1-m\le n}g_{n,m}(s)\sim \sum_{l_1\le n}g_{n,m}(s)\approx d_m(\alpha^{-1}|\varpi|^s)^m \cdot \sum_{n\ge l_1}g_n^1(s)$ where $$g_n^1(s)=\frac{(c\alpha^{-1}|\varpi|^{\frac{3}{2}})^n(c^{-1}\alpha^{-1}|\varpi|^{\frac{1}{2}+2s})}{1-c^{-1}\alpha^{-1}|\varpi|^{\frac{1}{2}+2s}}-\frac{(c\alpha^{-1}|\varpi|^{\frac{3}{2}+2s})^n(c^{-1}\alpha^{-1}|\varpi|^{\frac{1}{2}+2s})}{1-c^{-1}\alpha^{-1}|\varpi|^{\frac{1}{2}}}$$ and note that $\sum_{n \ge l_1}g_n^1(s) \sim 0.$ Thus the second term $\sum_{0\le m\le l_1-1}\sum_{l_1-m\le n}g_{n,m}(s) \sim 0$.\\

\noindent The third term $\sum_{l_1\le m}\sum_{0\le n }g_{n,m}(s)$ is $\sum_{l_1\le m}\sum_{0\le n<l_1 }g_{n,m}(s)+\sum_{l_1\le m}\sum_{l_1\le n }g_{n,m}(s)$. \\ For each $0\le n <l_1$, $\sum_{l_1\le m}g_{n,m}=\big(\sum_{l_1\le m}c_1(c\alpha^{-1}|\varpi|^{\frac{1}{2}+s})^m+c_2(c^{-1}\alpha^{-1}|\varpi|^{\frac{1}{2}+s})^m\big)\cdot g_n^2(s)$ where $$g_n^2(s)= (\frac{(c^{-1}\alpha^{-1}|\varpi|^{\frac{1}{2}})^n(c^{-1}\alpha^{-1}|\varpi|^{\frac{1}{2}+2s})}{1-c^{-1}\alpha^{-1}|\varpi|^{\frac{1}{2}+2s}}-\frac{(c^{-1}\alpha^{-1}|\varpi|^{\frac{1}{2}+2s})^n(c^{-1}\alpha^{-1}|\varpi|^{\frac{1}{2}})}{1-c^{-1}\alpha^{-1}|\varpi|^{\frac{1}{2}}}).$$ Thus $\sum_{l_1\le m}\sum_{0\le n<l_1 }g_{n,m}(s) \sim 0$ and $$\sum_{l_1\le m}\sum_{l_1\le n }g_{n,m}(s)\approx \big(\sum_{n \ge l_1}g_n^1(s)\big) \cdot \big(\sum_{l_1 \le m}c_1(c\alpha^{-1}|\varpi|^{\frac{1}{2}+s})^m+c_2(c^{-1}\alpha^{-1}|\varpi|^{\frac{1}{2}+s})^m\big)\sim0.$$ Since the above three components of $\sum_{m\ge0, n\ge0}g_{n,m}(s)$ are all $\sim 0$, $$\sum_{m\ge0, n\ge0}g_{n,m}(s) \sim 0.$$\\

\noindent Next we divide $\sum_{m\ge -n, n<0}g_{n,m}(s)=\sum_{m\ge -n, -l_1<n<0}g_{n,m}(s)+\sum_{m\ge -n, n \le -l_1}g_{n,m}(s).$ \\For each $-l_1<n<0$, $$\sum_{m\ge-n+l_1}g_{n,m}(s)=\big(\frac{c^{-1}\alpha^{-1}|\varpi|^{\frac{1}{2}+2s}}{1-c^{-1}\alpha^{-1}|\varpi|^{\frac{1}{2}+2s}}-\frac{c^{-1}\alpha^{-1}|\varpi|^{\frac{1}{2}}}{1-c^{-1}\alpha^{-1}|\varpi|^{\frac{1}{2}}}\big)k_n\cdot (c^{-1}\alpha^{-1}|\varpi|^{\frac{1}{2}})^n\cdot f_{n}(s)$$ where $k_n=\int_{F^3} \varphi(0,\varpi^{n}x_2,0) \cdot \varphi(x_1,x_2,x_3) dX$ and $$ f_n^1(s)=\frac{c_1(\alpha^{-1}|\varpi|^{\frac{1}{2}+s})^{-n+l_1}}{1-\alpha^{-1}|\varpi|^{\frac{1}{2}+s}}+\frac{c_2(c^{-2}\alpha^{-1}|\varpi|^{\frac{1}{2}+s})^{-n+l_1}}{1-c^{-2}\alpha^{-1}|\varpi|^{\frac{1}{2}+s}}.$$ Thus $$\sum_{m\ge-n,-l_1<n<0}g_{n,m}(s)=\sum_{-n\le m <-n+l_1,-l_1<n<0}g_{n,m}(s)+\sum_{m\ge-n+l_1,-l_1<n<0}g_{n,m}(s)\sim 0.$$\\

\noindent Next, we divide $$\sum_{m<-n,n<0}g_{n,m}(s)=\sum_{n<-l_1, 0\le m <-n-l_1}g_{n,m}(s)+\sum_{n<0, -n-l_1\le m <-n}g_{n,m}(s).$$ Again, $$\sum_{n<-l_1, 0\le m <-n-l_1}g_{n,m}(s)\sim \sum_{n\le -2l_1,0\le m <l_1}g_{n,m}(s)+\sum_{n\le-2l_1,l_1\le m <-n-l_1}g_{n,m}$$ and for each $0\le m<l_1$,
 $\sum_{n\le-2l_1}g_{n,m}=$ $$d_m(c|\varpi|^{-\frac{3}{2}-s})^m\cdot \big(\frac{(c^{-1}\alpha^{-1}|\varpi|^{\frac{1}{2}+2s})^{l_1}}{1-c^{-1}\alpha^{-1}|\varpi|^{\frac{1}{2}+2s}}-\frac{(c^{-1}\alpha^{-1}|\varpi|^{\frac{1}{2}})^{l_1}}{1-c^{-1}\alpha^{-1}|\varpi|^{\frac{1}{2}}}\big)\cdot \sum_{n\le-2l_1}(c^2|\varpi|^{-1-2s})^n \sim 0.$$
Note that $\sum_{n\le-2l_1,l_1\le m <-n-l_1}g_{n,m}=$ $$\big(\frac{(c^{-1}\alpha^{-1}|\varpi|^{\frac{1}{2}+2s})^{l_1}}{1-c^{-1}\alpha^{-1}|\varpi|^{\frac{1}{2}+2s}}-\frac{(c^{-1}\alpha^{-1}|\varpi|^{\frac{1}{2}})^{l_1}}{1-c^{-1}\alpha^{-1}|\varpi|^{\frac{1}{2}}}\big)\cdot \sum_{n\le-2l_1} f_n^2(s) \cdot(c^2|\varpi|^{-1-2s})^n$$ where $$f_n^2(s)=c_1 \cdot \frac{(c^2|\varpi|^{-1-s})^{l_1}-(c^2|\varpi|^{-1-s})^{-n-l_1}}{1-c^2|\varpi|^{-1-s}}+c_2\cdot \frac{(|\varpi|^{-1-s})^{l_1}-(|\varpi|^{-1-s})^{-n-l_1}}{1-|\varpi|^{-1-s}}.$$ Thus $\sum_{n\le-2l_1, l_1\le m<-n-l_1}g_{n,m}(s)\overset{1}{\sim} 0$ and so $\sum_{n<-l_1, 0\le m <-n-l_1}g_{n,m}(s)\overset{1}{\sim}0$.\\

\noindent To show $\sum_{n<0, -n-l_1\le m <-n}g_{n,m}(s)\overset{1}{\sim}0$, let $k=n+m$ and for each $-l_1\le k <0$, we will check $\sum_{n<0}g_{n,k-n}(s) \overset{1}{\sim}0.$ $$\sum_{n<0}g_{n,k-n}(s) \sim \sum_{n<-2l_1}g_{n,k-n}(s)\approx  (c_1\cdot \sum_{n<-2l_1}|\varpi|^{(-\frac{1}{2}-s)n}+c_2\cdot \sum_{n<-2l_1}(c^2|\varpi|^{(-\frac{1}{2}-s)})^n)$$ where $$f_n^3(s)=\frac{(c^{-1}\alpha^{-1}|\varpi|^{\frac{1}{2}+2s})^{l_1}}{1-c^{-1}\alpha^{-1}|\varpi|^{\frac{1}{2}+2s}}-\frac{(c^{-1}\alpha^{-1}|\varpi|^{\frac{1}{2}})^{l_1}}{1-c^{-1}\alpha^{-1}|\varpi|^{\frac{1}{2}}}.$$Thus $\sum_{n<0, -n-l_1\le m <-n}g_{n,m}(s)\overset{1}{\sim}0$ and so we checked $$\sum_{l\ge l_1, l >(n+m)}\big( (c^{-1}\alpha^{-1}|\varpi|^{\frac{1}{2}+2s})^{l} \cdot  |\varpi|^{-s(2n+m)}-(c^{-1}\alpha^{-1}|\varpi|^{\frac{1}{2}})^{l}  \cdot |\varpi|^{s(|n+m|+|n|)}\big) \overset{1}{\sim} 0.$$\\

\noindent Next we turn to show $$\sum_{n\in \Z , m\ge0}c^{2n+m}|\varpi|^{n-\frac{m}{2}}d_ma_{n,m}\cdot \Big(\sum_{l\ge l_1, n< l \le n+m}(c^{-1}\alpha^{-1}|\varpi|^{\frac{1}{2}})^{l} \cdot  (|\varpi|^{sm}-|\varpi|^{s(|n+m|+|n|)})\Big)\sim 0.$$ It equals $\sum_{m\ge0}c^m d_m|\varpi|^{-\frac{m}{2}}\sum_{n\ge0}f_{n,m}(s)$ where $$f_{n,m}(s)=c^{2n}|\varpi|^{n}a_{n,m}\cdot \Big(\sum_{l\ge l_1, n< l \le n+m}(c^{-1}\alpha^{-1}|\varpi|^{\frac{1}{2}})^{l} \cdot  (|\varpi|^{sm}-|\varpi|^{s(2n+m)})\Big).$$ Then $$\sum_{n\ge0}f_{n,m}(s)=\sum_{0\le n\le l_1-1}f_{n,m}(s)+\sum_{l_1 \le n}f_{n,m}(s)$$ and $$\sum_{m\ge 0}c^m d_m|\varpi|^{-\frac{m}{2}}\big(\sum_{0\le n\le l_1-1}f_{n,m}(s)\big)\sim \sum_{m\ge 2l_1}c^m d_m|\varpi|^{-\frac{m}{2}}\big(\sum_{0\le n\le l_1-1}f_{n,m}(s)\big)=$$  $$\sum_{0\le n\le l_1-1}c^{2n}|\varpi|^{n}(1-|\varpi|^{2ns})\sum_{m\ge2 l_1}c^m d_m|\varpi|^{(s-\frac{1}{2})m}a_{n,m}\cdot \frac{(c^{-1}\alpha^{-1}|\varpi|^{\frac{1}{2}})^{l_1}-(c^{-1}\alpha^{-1}|\varpi|^{\frac{1}{2}})^{n+m+1}}{1-c^{-1}\alpha^{-1}|\varpi|^{\frac{1}{2}}}.$$

\noindent For each $0\le n \le l_1-1,$ $$\sum_{m\ge2 l_1}c^m d_m|\varpi|^{(s-\frac{1}{2})m}a_{n,m}\cdot \frac{(c^{-1}\alpha^{-1}|\varpi|^{\frac{1}{2}})^{l_1}-(c^{-1}\alpha^{-1}|\varpi|^{\frac{1}{2}})^{n+m+1}}{1-c^{-1}\alpha^{-1}|\varpi|^{\frac{1}{2}}} \overset{1}{\sim} 0$$ and so $$\sum_{m\ge 0}c^m d_m|\varpi|^{-\frac{m}{2}}\big(\sum_{0\le n\le l_1-1}f_{n,m}(s)\big)\overset{1}{\sim} 0.$$

\noindent For each $m \in \N$, $$\sum_{n\ge l_1}f_{n,m}\approx|\varpi|^{sm}\big(1-(c^{-1}\alpha^{-1}|\varpi|^{\frac{1}{2}})^m\big)\cdot \big(\sum_{n \ge l_1}(c\alpha^{-1}|\varpi|^{\frac{3}{2}})^n-(c\alpha^{-1}|\varpi|^{\frac{3}{2}+2s})^n\big).$$ Thus $\sum_{m\ge0}c^m d_m|\varpi|^{-\frac{m}{2}}\sum_{n\ge l_1}f_{n,m} \overset{1}{\sim} 0$ and so we showed $$\sum_{n\in \Z , m\ge0}c^{2n+m}|\varpi|^{n-\frac{m}{2}}d_ma_{n,m}\cdot \Big(\sum_{l\ge l_1, n< l \le n+m}(c^{-1}\alpha^{-1}|\varpi|^{\frac{1}{2}})^{l} \cdot  (|\varpi|^{sm}-|\varpi|^{s(|n+m|+|n|)})\Big) \overset{1}{\sim}0.$$\\

\noindent Finally, we investigate the last sum
\begin{equation}\label{lastsum}\sum_{n\in \Z , m\ge0}c^{2n+m}|\varpi|^{(1+2s)n+(s-\frac{1}{2})m}d_ma_{n,m}\cdot \Big(\sum_{ l_1\le l\le n}\big( c^{-1}\alpha^{-1}|\varpi|^{\frac{1}{2}-2s})^{l}-(c^{-1}\alpha^{-1}|\varpi|^{\frac{1}{2}})^{l}\Big).\end{equation}
\noindent It equals $$k_1^1\cdot \sum_{m\ge0}d_m(c|\varpi|^{s-\frac{1}{2}})^m\big(\sum_{n\ge l_1}(c^2|\varpi|^{1+2s})^n\cdot g_n(s)\big)\text{ where }$$ $g_n(s)=$ $$\frac{(c^{-1}\alpha^{-1}|\varpi|^{\frac{1}{2}-2s})^{l_1}-(c^{-1}\alpha^{-1}|\varpi|^{\frac{1}{2}-2s})^{n+1}}{1-c^{-1}\alpha^{-1}|\varpi|^{\frac{1}{2}-2s}}-\frac{(c^{-1}\alpha^{-1}|\varpi|^{\frac{1}{2}})^{l_1}-(c^{-1}\alpha^{-1}|\varpi|^{\frac{1}{2}})^{n+1}}{1-c^{-1}\alpha^{-1}|\varpi|^{\frac{1}{2}}}.$$ Thus $\sum_{n\ge l_1}(c^2|\varpi|^{1+2s})^n\cdot g_n(s) \sim 0$ and $\sum_{m\ge 0}d_m(c|\varpi|^{s-\frac{1}{2}})^m \overset{2}{\sim}0$, and so we see that \ref{lastsum} $\overset{1}{\sim}0.$\\
\noindent We have checked (\ref{3})$\overset{1}{\sim} 0$.\\

\noindent Putting all these things together, we verified our claim (\ref{lem2}).

\end{proof}

\begin{remark}\label{5.2}The assumption that $F$ is totally real and $E$ is totally complex in our Theorem \ref{thm} is not so crucial and can be dropped out. For archimedean places $v$ which splits in $E$, we can conduct the similar computation as in the $p$-adic case because we have (\ref{height2}) at hands and the asymptotic behavior of various matrix coefficient are of the similar shape. For brevity, we omit the details.

\end{remark}

\section*{Acknowledgements}
The author expresses deep gratitude to his advisor professor Haseo Ki for his constant support and encouragement during this work. We also thank to professor Atsushi Ichino for his suggestion this problem and many helpful discussion. As being evident to the reader, this paper owes much to Shunsuke Yamana's result. So I would like to thank to him for making his preprint available as well as answering my question. This work was supported by the National Research Foundation of Korea(NRF) grant funded by the Korea government(MSIP)(ASARC, NRF-2007-0056093).

\end{document}